\documentclass{article}
\usepackage[utf8]{inputenc}
\usepackage[margin=1in]{geometry}
\usepackage{graphicx}
\usepackage{amsmath}
\usepackage{amssymb}
\usepackage{amsthm}
\usepackage{xcolor}
\usepackage[parfill]{parskip}
\usepackage{tikz}
\usetikzlibrary {math,arrows.meta,decorations.pathreplacing,decorations.shapes}
\usepackage{mathtools}
\usepackage{hyperref}
\usepackage{caption}
\usepackage{enumitem}
\usepackage{biblatex}
\usepackage{multirow}
\usepackage{multicol}
\usepackage{nicematrix}
\usepackage{float}
\usepackage{mathdots}
\usepackage{listings}
\newtheorem{theorem}{Theorem}[section]
\newtheorem{proposition}[theorem]{Proposition}
\newtheorem{lemma}[theorem]{Lemma}
\newtheorem{corollary}[theorem]{Corollary}
\theoremstyle{definition}
\newtheorem{definition}[theorem]{Definition}
\newtheorem{observation}[theorem]{Observation}

\newtheorem{conjecture}[theorem]{Conjecture}

\newtheorem{remark}[theorem]{Remark}

\newtheorem{construction}[theorem]{Construction}

\title{Fat Shellable Spheres}
\author{Joshua Hinman\\
\small Department of Mathematics\\[-0.8ex]
\small University of Washington\\[-0.8ex]
\small Seattle, WA 98195-4350, USA\\[-0.8ex]
\small \texttt{joshrh@uw.edu}}
\date{}

\begin{document}
\setcounter{MaxMatrixCols}{100}
\NiceMatrixOptions{cell-space-limits=1pt}
\maketitle
\begin{abstract}
    The fatness of a 4-polytope or 3-sphere is defined as $(f_1+f_2-20)/(f_0+f_3-10)$. We construct arbitrarily fat, strongly regular CW 3-spheres that are both shellable and dual shellable. These spheres have $f$-vectors $(\Theta(n),\Theta(n\alpha(n)),\Theta(n\alpha(n)),\Theta(n))$, where $\alpha$ is the inverse Ackermann function.
\end{abstract}


\section{Introduction}

Ziegler writes that ``it would be nice to have arbitrarily fat shellable 3-spheres" \cite{eppstein25}. In this paper, we construct arbitrarily fat shellable 3-spheres.

To explain why the above would be nice, we first consider $f$-vectors of polytopes. If $P$ is a $d$-polytope, then the $f$-vector of $P$ is defined as $(f_0,\ldots,f_{d-1})$, where $f_k$ represents the number of $k$-dimensional faces of $P$ for $k=0,\ldots,d-1$. Herefrom arises the guiding question of much polyhedral combinatorics: which integer vectors are the $f$-vector of a polytope?

Steinitz answered this question for 3-polytopes in 1906 \cite{steinitz06}. An integer vector $(f_0,f_1,f_2)$ is the $f$-vector of a 3-polytope if and only if the following hold:
\begin{itemize}
    \item $f_0 - f_1 + f_2 = 2$ (Euler's relation),
    \item $2f_1 - 3f_0 \geq 0$ (tight for simple polytopes),
    \item $2f_1 - 3f_2 \geq 0$ (tight for simplicial polytopes).
\end{itemize}
The above define a two-dimensional cone in $\mathbb{R}^3$ with apex $(4,6,4)$, the $f$-vector of the 3-simplex. Note that the $d$-simplex in general has $f$-vector $(d+1,\binom{d+1}{2},\ldots,\binom{d+1}{d})$.

A century later, we are nowhere near characterizing the $f$-vectors of 4-polytopes. We therefore ask a simpler question: what is the \emph{$f$-vector cone}, or the closure of the convex cone spanned by all $f$-vectors in $\mathbb{Z}^4$ with apex the $f$-vector $(5,10,10,5)$ of the 4-simplex? All 4-polytopes must satisfy the following (see Bayer \cite{bayer87}):
\begin{itemize}
    \item $f_0 - f_1 + f_2 - f_3 = 0$ (Euler's relation),
    \item $f_0 \geq 5$ (tight for simplex),
    \item $f_3 \geq 5$ (tight for simplex),
    \item $f_1 - 2f_0 \geq 0$ (tight for simple polytopes),
    \item $f_2 - 2f_3 \geq 0$ (tight for simplicial polytopes),
    \item $5f_0 - 2f_1 - 2f_2 + 5f_3 \leq 10$ (proven by Bayer; tight for stacked polytopes and their duals).
\end{itemize}
The above define a three-dimensional cone in $\mathbb{R}^4$ with apex $(5,10,10,5)$. However, we do not know if the set of $f$-vectors spans its whole interior. Here the \emph{fatness} parameter, defined as $(f_1+f_2-20)/(f_0+f_3-10)$, comes into play. If there exist arbitrarily fat 4-polytopes, then this cone is exactly the $f$-vector cone; if fatness of 4-polytopes is bounded, then the $f$-vector cone is smaller. See Ziegler \cite{ziegler20}\cite{ziegler02} for further discussion.

Eppstein, Kuperberg, and Ziegler introduced the fatness parameter in 2003 and constructed a 4-polytope with fatness $\sim\!5.048$, the highest then known. Eppstein et al. further constructed a family of strongly regular CW 3-spheres with unbounded fatness, having $f$-vectors $(\Theta(n^{12}),\Theta(n^{13}),\Theta(n^{13}),\Theta(n^{12}))$. The latter construction was probabilistic and involved subdividing a high-genus surface into a CW complex with many edges, then thickening it to obtain a three-dimensional complex with many edges and ridges. \cite{eppstein03}

In 2006, Paffenholz constructed a family of polytopes with fatness $6-\varepsilon$ for arbitrarily small $\varepsilon$. Paffenholz's ``E-construction" involved stellating a simple polytope $P$ and merging any two of the resulting facets that share a ridge of $P$ \cite{paffenholz06}. Around the same time, Ziegler constructed a family of polytopes with fatness $9-\varepsilon$, the highest to date. Ziegler's polytopes arise from taking a ``deformed product" of many large polygons, then projecting the product into $\mathbb{R}^4$ while preserving all edges and all large polygon faces \cite{ziegler04}.

A natural approach to constructing fat polytopes is to construct fat 3-spheres first, then attempt to find polytopes that realize them. However, for a sphere to have any hope of polytopality, it must be both shellable and dual shellable. It is unknown whether the probabilistic construction of Eppstein et al. can be specified to yield shellable spheres \cite{eppstein25}. We construct an explicit family of arbitrarily fat, strongly regular CW 3-spheres that are both shellable and dual shellable. Our spheres have $f$-vectors $(\Theta(n),\Theta(n\alpha(n)),\Theta(n\alpha(n)),\Theta(n))$, where $\alpha$ is the inverse Ackermann function.

We generate our family of spheres recursively by gluing together small CW complexes into large ones. Our method is based on recursive constructions of Davenport-Schinzel sequences (see Sharir and Agarwal \cite{sharir95}) and a construction of pattern-avoiding binary matrices due to F\"uredi and Hajnal \cite{furedi92}. We ensure shellability by building CW complexes in a ``bird's-eye view" where each facet has a well-defined top and bottom, making the list of facets from lowest to highest a shelling order.

The structure of our paper is as follows. In Section \ref{s-preliminaries}, we introduce the main concepts of our paper related to polytopes and CW complexes. Section \ref{s-binary-matrices} is a warm-up: we show a modified version of F\"uredi and Hajnal's matrix construction and use it to generate polytopal complexes with high \emph{complexity}, a related parameter to fatness. In Section \ref{s-main-section}, we give our main construction of fat, shellable, dual shellable 3-spheres. We conjecture that these spheres are polytopal.


\section{Preliminaries}
\label{s-preliminaries}

In this section, we introduce the central concepts of our paper. We discuss polytopes and polytopal complexes (\S\ref{ss-polytopes}), CW complexes (\S\ref{ss-cw}), shellings and dual shellings (\S\ref{ss-shellings}), and the fatness parameter (\S\ref{ss-fatness}).

\subsection{Polytopes}
\label{ss-polytopes}

We begin with a brief overview of polytopes and polytopal complexes. We refer the reader to \cite{grunbaum03}\cite{,ziegler95} for any undefined terminology.

\begin{definition}
    A \emph{$d$-polytope} is the convex hull of finitely many points in $\mathbb{R}^d$ with affine hull $\mathbb{R}^d$.
\end{definition}

\begin{definition}
    Let $P$ be a $d$-polytope. A \emph{face} of $P$ is either $P$ itself or a polytope of the form $P \cap H$, where $H$ is a hyperplane that does not intersect the interior of $P$. Faces of dimension $k$ are called \emph{$k$-faces}; 0-, 1-, $(d-2)$-, and $(d-1)$-faces are called \emph{vertices}, \emph{edges}, \emph{ridges}, and \emph{facets}, respectively.
\end{definition}

We regard the empty set as a $(-1)$-polytope.

\begin{definition}
    A \emph{polytopal $d$-complex} is a collection $Q$ of polytopes in $\mathbb{R}^d$ such that
    \begin{itemize}
        \item the empty set is in $Q$,
        \item if $P$ is a polytope in $Q$, then every face of $P$ is in $Q$, and
        \item if $P$ and $P'$ are polytopes in $Q$, then $P \cap P'$ is a (possibly empty) face of both $P$ and $P'$.
    \end{itemize}
\end{definition}

\begin{definition}
    Let $P$ be a polytope. The \emph{boundary complex} of $P$, denoted $\partial P$, is the polytopal complex consisting of all proper faces of $P$.
\end{definition}

\subsection{CW complexes}
\label{ss-cw}

We proceed with an overview of CW complexes. Undefined terminology may be found in \cite{bjorner84}\cite{bjorner83}.

\begin{definition}
    For $d \geq 0$, a \emph{CW $d$-complex} is a topological space defined recursively as follows.
    \begin{itemize}
        \item A CW 0-complex is a finite set of points under the discrete topology.
        \item For $d>0$, a CW $d$-complex $X$ is obtained from a CW $(d-1)$-complex $X'$ by attaching closed $d$-balls $Z_1,\ldots,Z_n$ via gluing maps $\partial Z_1 \to X', \ldots, \partial Z_n \to X'$.
    \end{itemize}
    The \emph{faces} of $X$ are the closed balls used to construct $X$, including initial points, plus the empty set. Faces of dimension $k$ are called \emph{$k$-faces}; 0-, 1-, $(d-1)$-, and $d$-faces are called \emph{vertices}, \emph{edges}, \emph{ridges}, and \emph{facets}, respectively.
\end{definition}

\begin{definition}
    A CW complex $X$ is \emph{regular} if every gluing map used in its construction is a homeomorphism. $X$ is \emph{strongly regular} if, in addition, the intersection of any two faces of $X$ is itself a face of $X$.
\end{definition}

\begin{definition}
    A CW complex $X$ is \emph{pure} if every face of $X$ is contained in a facet of $X$.
\end{definition}

\begin{definition}
    A \emph{CW $d$-sphere} (resp. \emph{$d$-ball}) is a pure CW $d$-complex homeomorphic to a $d$-sphere (resp. $d$-ball).
\end{definition}

The boundary complex of any $(d+1)$-polytope is a strongly regular CW $d$-sphere. The converse does not hold in general but does when $d=3$ due to the famous 1922 theorem of Steinitz \cite{steinitz22}. See also \cite[\S 4]{ziegler95}.

\begin{theorem}[Steinitz]
\label{Steinitz}
    Any strongly regular CW 2-sphere is isomorphic to the boundary complex of a 3-polytope.
\end{theorem}

We will sometimes want to discuss the local structure of a CW complex around a specific face. It will then be useful to consider the relative interior, star, and vertex figure, which we define below.

\begin{definition}
    Let $Z$ be a face of some CW complex. The \emph{relative interior} of $Z$, denoted $\operatorname{relint}Z$, is the set $Z\backslash\partial Z$.
\end{definition}

\begin{definition}
    Let $X$ be a pure, regular CW complex and $Z$ a face of $X$. The \emph{star} of $Z$, denoted $\operatorname{st}Z$, is the union of the relative interiors of all faces of $X$ containing $Z$.
\end{definition}

\begin{definition}
    Let $X$ be a pure, regular CW $d$-complex and $v$ a vertex of $X$. Let $U$ be an open neighborhood of $v$ in $X$ with $\operatorname{cl}U \subset \operatorname{st}v$ and with $Z \cap \operatorname{cl}U$ homeomorphic to a closed $k$-ball for all $k$-faces $Z\subseteq X$ containing $v$. The \emph{vertex figure} of $v$, denoted $X/v$, is the CW $(d-1)$-complex with $(k-1)$-faces $Z/v\coloneq Z \cap \operatorname{cl}U\backslash U$ for $k=1,\ldots,d$ and for $k$-faces $Z\subseteq X$ containing $v$.
\end{definition}

\begin{observation}
\label{pl}
    Suppose $X$ is a pure, regular CW complex and $v$ a vertex of $X$. If $X$ is strongly regular, then $X/v$ is strongly regular. If $X$ is a CW $3$-sphere, then $X/v$ is a CW $2$-sphere. If $X$ is a CW $3$-ball and $v \in \operatorname{int}X$ or $v \in \partial X$, then $X/v$ is respectively a CW $2$-sphere or $2$-ball (see \cite[p.~278]{Bjorner00}).
\end{observation}

Our main work in this paper is on obtaining new CW complexes from old. The tools for this include subcomplexes and coning, defined next.

\begin{definition}
    Let $X$ be a CW complex. A \emph{subcomplex} of $X$ is a CW complex $W \subseteq X$ such that every face of $W$ is a face of $X$, and the set of faces of $W$ is closed under containment in $X$.
\end{definition}

\begin{definition}
    Let $X$ be a regular CW complex, $W$ a subcomplex of $X$, and $v$ a point outside $X$. To \emph{cone} over $W$ at $v$ is an operation wherein we add $v$ as a vertex of $X$, then attach to each nonempty $k$-face $Z \subseteq W$ a new $(k+1)$-face isomorphic to a pyramid with base $Z$ and apex $v$.
\end{definition}

\begin{observation}
    If we cone a strongly regular CW complex over any subcomplex, we obtain another strongly regular CW complex. If we cone a CW $d$-ball over its boundary, we obtain a CW $d$-sphere.
\end{observation}

Hereafter, when we refer to any CW complex, ``strongly regular" will be implied unless otherwise indicated. We will typically also omit the word ``CW" and write \emph{$d$-sphere} to mean ``strongly regular CW $d$-sphere"; likewise \emph{$d$-ball}.

\subsection{Shellings}
\label{ss-shellings}

We now discuss shellings and dual shellings of polytopes and CW complexes. We use the definition of shellability given by Bj\"orner and Wachs in \cite{bjorner83}.

\begin{definition}
    Let $X$ be a pure CW $d$-complex with facets $Y_1, \ldots, Y_n$. The ordered set $(Y_1,\ldots,Y_n)$ is a \emph{shelling} of $X$ if either $d=0$, or $d>0$ and
    \begin{itemize}
        \item there exists a shelling of $\partial Y_1$,
        \item for $i=2, \ldots, n$, the intersection $Y_i \cap \bigcup_{j=1}^{i-1} Y_j$ is a pure $(d-1)$-complex, and there exists a shelling of $\partial Y_i$ beginning with the $(d-1)$-faces of $Y_i \cap \bigcup_{j=1}^{i-1} Y_j$.
    \end{itemize}
    We call $X$ \emph{shellable} if it admits a shelling.
\end{definition}

\begin{definition}
    Let $X$ be a pure CW $d$-complex with vertices $v_1,\ldots,v_n$. The ordered set $(v_1, \ldots, v_n)$ is a \emph{dual shelling} of $X$ if either $d=0$, or $d>0$ and
    \begin{itemize}
        \item there exists a dual shelling of $X/v_1$,
        \item for $i=2,\ldots,n$, the set $X/v_i \cap \bigcup_{j=1}^{i-1}\operatorname{st}v_j$ is a nonempty union of vertices of $X/v_i$ and their stars, and there exists a dual shelling of $X/v_i$ beginning with the vertices in $X/v_i \cap \bigcup_{j=1}^{i-1}\operatorname{st}v_j$.
    \end{itemize}
    We call $X$ \emph{dual shellable} if it admits a dual shelling.
\end{definition}

Bj\"orner and Wachs proved that a CW complex is shellable (respectively, dual shellable) if and only if its face poset is \emph{dual CL-shellable} (respectively, \emph{CL-shellable}) \cite[Corollary 4.4]{bjorner83}.

\begin{theorem}[Bruggesser and Mani]
\label{line-shelling}
    The boundary complex of any polytope is shellable.
\end{theorem}

Bruggesser and Mani proved Theorem \ref{line-shelling} in \cite{bruggesser71} using the technique of \emph{line shellings}. To obtain a line shelling, we move away from our polytope in a straight line and begin listing facets in the order they become visible. We then approach along the same line, from the opposite direction, and list the remaining facets in the order they become obscured.

Under polarization, a shelling of any polytope becomes a dual shelling of the dual polytope. A line shelling becomes a \emph{sweep}, a dual shelling induced by a linear functional \cite[p.~142a]{grunbaum03}.

\begin{corollary}
\label{vertex-star}
    Let $P$ be a polytope and $v$ a vertex of $P$. Then there exists a shelling of $\partial P$ beginning (or ending) with the facets of $P$ that contain $v$.
\end{corollary}

The above follows from Bruggesser and Mani's line shellings; see \cite[Proposition 2]{bruggesser71}\cite[Corollary 8.13]{ziegler95}.

\begin{corollary}
\label{shel-dual-shel}
    The boundary complex of any $(d+1)$-polytope is a shellable, dual shellable $d$-sphere.
\end{corollary}

The following three lemmas concern the shelling properties of 2- and 3-spheres. We include their proofs for completeness.

\begin{lemma}
\label{2-sphere-shelling}
    Let $S$ be a 2-sphere, and let $W$ be the union of a proper, nonempty subset of the facets of $S$. Then there exists a shelling of $S$ beginning with the facets of $W$ if and only if $W$ is homeomorphic to a closed disc.
\end{lemma}

\begin{proof}
    If there exists a shelling of $S$ beginning with the facets of $W$, then $W$ is homeomorphic to a closed disc by \cite[Proposition 4.3]{bjorner84}.

    Conversely, suppose $W$ is homeomorphic to a closed disc. Let $Q$ be the 2-sphere obtained from $W$ by adding a new vertex $v$ and coning over the boundary of $W$ at $v$. Let $Q'$ be the 2-sphere obtained from $\operatorname{cl}(S\backslash W)$ by adding a new vertex $v'$ and coning over the boundary of $\operatorname{cl}(S\backslash W)$ at $v'$. By Theorem \ref{Steinitz}, $Q$ and $Q'$ are each isomorphic to the boundary complex of some 3-polytope.

    By Corollary \ref{vertex-star}, there exists a shelling $(T_1, \ldots, T_m, \ldots, T_n)$ of $Q$ such that $T_1, \ldots, T_m$ are the facets of $W$. Also by Corollary \ref{vertex-star}, there exists a shelling $(T'_1, \ldots, T'_{m'}, \ldots, T'_{n'})$ of $Q'$ such that $T'_{m'+1}, \ldots, T'_{n'}$ are the facets of $\operatorname{cl}(S\backslash W)$.

    Since $(T_1,\ldots,T_n)$ is a shelling of $Q$, for $i=2,\ldots,m$, there exists a shelling of $\partial T_i$ beginning with the edges of $T_i\cap\bigcup_{j=1}^{i-1}T_j$. Meanwhile, since $(T'_1, \ldots, T'_{n'})$ is a shelling of $Q'$, for $i=m'+1, \ldots, n'$, there exists a shelling of $\partial T'_i$ beginning with the edges of
    \[
        T'_i \cap \bigcup_{j=1}^{i-1}T'_j = T'_i \cap \left(\bigcup_{j=1}^m T_j \cup \bigcup_{j=m'+1}^{i-1}T'_j\right).
    \]
    Thus, by definition, $(T_1, \ldots, T_m, T'_{m'+1}, \ldots, T'_{n'})$ is a shelling of $S$ beginning with the facets of $W$.
\end{proof}

The next lemma follows by duality.

\begin{lemma}
\label{2-sphere-dual-shelling}
    Let $S$ be a 2-sphere, and let $U$ be the union of a proper, nonempty subset of the vertices of $S$ and their stars. Then there exists a dual shelling of $S$ beginning with the vertices in $U$ if and only if $U$ is homeomorphic to an open disc.
\end{lemma}

From Lemmas \ref{2-sphere-shelling}-\ref{2-sphere-dual-shelling}, we obtain the following.

\begin{lemma}
\label{3-sphere-shelling}
    Let $S$ be a 3-sphere with vertices $v_1,\ldots,v_n$ and facets $T_1,\ldots,T_m$.
    \begin{enumerate}[label=\roman*.]
        \item $(T_1,\ldots,T_m)$ is a shelling of $S$ if and only $T_i \cap \bigcup_{j=1}^{i-1}T_j$ is homeomorphic to a closed disc for $i=2,\ldots,m-1$.
        \item $(v_1,\ldots,v_n)$ is a dual shelling of $S$ if and only if $X/v_i \cap \bigcup_{j=1}^{i-1}\operatorname{st}v_j$ is a union of vertices of $X/v_i$ and their stars, and is homeomorphic to an open disc, for $i=2,\ldots,n-1$.
    \end{enumerate}
\end{lemma}

\subsection{Fatness}
\label{ss-fatness}

Lastly, we introduce $f$-vectors and the fatness and complexity parameters for polytopes and complexes.

\begin{definition}
    Let $X$ be a $(d+1)$-polytope, polytopal $d$-complex, or CW $d$-complex. We define $f_k(X)$ as the number of $k$-faces of $X$ for $k=0, \ldots, d$. We define the \emph{$f$-vector} of $X$ as $f(X) = (f_0(X),\ldots,f_d(X))$.
\end{definition}

\begin{definition}
    Let $X$ be a 4-polytope, polytopal 3-complex, or CW 3-complex with $f(X)=(f_0,f_1,f_2,f_3)$. The \emph{fatness} of $X$ is defined as
    \[
        \mathcal{F}(X)=\frac{f_1+f_2-20}{f_0+f_3-10}.
    \]
\end{definition}

\begin{definition}
    Let $X$ be a 4-polytope, polytopal 3-complex, or CW 3-complex with $f(X)=(f_0,f_1,f_2,f_3)$. Let $f_{03}$ be the number of vertex-facet incidences in $X$. The \emph{complexity} of $X$ is defined as
    \[
        \mathcal{C}(X)=\frac{f_{03}-20}{f_0+f_3-10}.
    \]
\end{definition}

The fatness and complexity of spheres are asymptotically equivalent. Ziegler gives the following in \cite{ziegler02}.

\begin{proposition}[Ziegler]
    For all 3-spheres $S$ not isomorphic to the boundary of the 4-simplex,
    \[
        \mathcal{C}(S) \leq 2\mathcal{F}(S)-2,\qquad\mathcal{F}(S) \leq 2\mathcal{C}(S)-2.
    \]
\end{proposition}

The next two theorems represent the state of the art on fat polytopes and spheres.

\begin{theorem}[Ziegler]
    There exist 4-polytopes with fatness $9-\varepsilon$ for arbitrarily small $\varepsilon$.
\end{theorem}

Ziegler constructs these polytopes in \cite{ziegler04} by taking a ``deformed product" of many large, even-sided polygons and projecting it into $\mathbb{R}^4$. These polytopes have fatness approaching 9 and complexity approaching 16.

\begin{theorem}[Eppstein, Kuperberg, and Ziegler]
    There exist arbitrarily fat 3-spheres.
\end{theorem}

Eppstein et al. construct these spheres in \cite{eppstein03} from randomly chosen, finite covering spaces of high-genus surfaces. It is unknown whether these spheres can be specified to be shellable \cite{eppstein25}.


\section{Warm-up: binary matrices}
\label{s-binary-matrices}

In this section, we motivate our main construction with an analogue in the setting of binary matrices. We construct a family of $n \times n$ binary matrices avoiding the pattern
\[
    \begin{bmatrix}
        \ & 1 & \ & 1 \\
        1 & \ & 1 & \
    \end{bmatrix},
\]
whose total number of ones is superlinear in $n$. F\"uredi and Hajnal gave the first such construction in \cite{furedi92}; ours is broadly the same, but we make some small adjustments for accuracy, and we include for completeness a proof of the superlinear bound.

We first introduce the Ackermann function (\S\ref{ss-ackermann}), which we will use to compute asymptotic bounds both here and in our main construction of spheres. We then give the aforementioned construction of pattern-avoiding matrices (\S\ref{ss-matrix}). Finally, we use these matrices to generate a family of polytopal 3-complexes with unbounded complexity (\S\ref{ss-moment-curve}).

\subsection{The Ackermann function}
\label{ss-ackermann}

The objects we will construct are huge, and we cannot describe their asymptotic behavior with polynomials, exponentials, or similar. Thus, we introduce the extremely fast-growing Ackermann function $A(s,t)$ and a related function $C(s,t)$. Several distinct versions of both functions appear in the literature; we define $A(s,t)$ and $C(s,t)$ as in \cite{sharir95} and \cite{furedi92}, respectively. Our discussion in this subsection parallels \cite[pp.~21-22]{sharir95}.

\begin{definition}
    For positive integers $s$ and $t$, the \emph{Ackermann function} $A(s,t)$ is defined recursively as follows.
    \begin{itemize}
        \item $A(1,t) = 2t$ for all $t$.
        \item $A(s,1) = 2$ for all $s$.
        \item $A(s,t) = A(s-1,A(s,t-1))$ for $s,t>1$.
    \end{itemize}
    The \emph{inverse Ackermann function} is defined as $\alpha(n) = \min\{s \mid A(s,s) \geq n\}$.
\end{definition}

\begin{definition}
    For positive integers $s$ and $t$, the function $C(s,t)$ is defined recursively as follows.
    \begin{itemize}
        \item $C(1,t)=1$ for all $t$.
        \item $C(s,1)=2$ for all $s>1$.
        \item $C(s,t)=C(s,t-1)C(s-1,C(s,t-1))$ for $s,t>1$.
    \end{itemize}
\end{definition}

Values of $A(s,t)$ and $C(s,t)$ for small $s$ and $t$ are given in Table \ref{ackermann}.

\begin{table}[t]
\[
    \begin{NiceArray}[t,first-row]{|c|llll|}
        \multicolumn{5}{c}{A(s,t)}\\
        \hline
        \diagbox{s}{t}& 1 & 2 & 3 & 4\\
        \hline
        1 & 2 & 4 & 6 & 8\\
        2 & 2 & 4 & 8 & 16\\
        3 & 2 & 4 & 16 & 2^{16}\\
        4 & 2 & 4 & 2^{16} & \underbrace{2^{2^{\iddots{}^{{}^{{}^2}}}}}_\text{\tiny $2^{16}$ twos}\\
        \hline
    \end{NiceArray}
    \qquad
    \begin{NiceArray}[t,first-row]{|c|llll|}
        \multicolumn{5}{c}{C(s,t)}\\
        \hline
        \diagbox{s}{t}& 1 & 2 & 3 & 4\\
        \hline
        1 & 1 & 1 & 1 & 1\\
        2 & 2 & 2 & 2 & 2\\
        3 & 2 & 4 & 8 & 16\\
        4 & 2 & 8 & 2^{11} & 2^{11+2^{11}}\\
        \hline
    \end{NiceArray}
\]
    \caption{Values of $A(s,t)$ and $C(s,t)$ for $s,t < 5$.}
    \label{ackermann}
\end{table}

\begin{observation}
    The following are easy to check.
    \[
        A(2,t)=2^t, \qquad A(s,2) = 4, \qquad C(2,t) = 2, \qquad C(3,t) = 2^t, \qquad C(s,2) = 2^{s-1}.
    \]
\end{observation}

\begin{lemma}
    \label{ackermann-properties}
    Let $s$ and $t$ be positive integers.
    \begin{enumerate}[label=\roman*.]
        \item $A(s,t+1) > A(s,t)$.
        \item $A(s,t) \geq t+2$ for $t > 1$.
        \item $A(s+1,t) \geq A(s,t+1)$ for $t > 2$.
        \item $A(s,3) \geq 2^{s+1}$.
    \end{enumerate}
\end{lemma}

\begin{proof}
    We will prove i-iv in order.
    \begin{enumerate}[wide,label=\roman*.]
        \item By definition, $A(1,t+1) = 2(t+1) > 2t = A(1,t)$. We also know that $A(s,2) = 4 > 2 = A(s,1)$. The cases $s,t>1$ follow by induction:
        \[
            A(s,t+1) = A(s-1,A(s,t)) > A(s-1,A(s,t-1)) = A(s,t).
        \]
        \item If $t=2$, we have $A(s,2) = 4 = t+2$. The cases $t>2$ then follow from i.
        \item Let $t > 2$. Then by ii and the recursive definition of $A(s,t)$,
        \[
            A(s+1,t) = A(s,A(s+1,t-1)) \geq A(s,t+1).
        \]
        \item If $s=1$, we have $A(1,3) = 6 > 4 = 2^{s+1}$. If $s>1$, then by iii,
        \[
            A(s,3) \geq A(s-1,4) \geq A(s-2,5) \geq \cdots \geq A(2,s+1) = 2^{s+1}.\qedhere
        \]
    \end{enumerate}
\end{proof}

\begin{lemma}
    \label{c-properties}
    Let $s$ and $t$ be positive integers.
    \begin{enumerate}[label=\roman*.]
        \item $C(s,t+1) \geq C(s,t)$.
        \item $C(s,t+1) \geq 2C(s,t)$ for $s>2$.
    \end{enumerate}
\end{lemma}

\begin{proof}
    We will prove i and then ii.
    \begin{enumerate}[wide,label=\roman*.]
        \item By definition, $C(1,t+1) = 1 = C(1,t)$. If $s>1$, then
        \[
            C(s,t+1) = C(s,t)C(s-1,C(s,t)) \geq C(s,t).
        \]
        \item Let $s>2$. It follows from i that $C(s-1,t) \geq C(s-1,1) = 2$ for all $t$. Thus,
        \[
            C(s,t+1) = C(s,t)C(s-1,C(s,t)) \geq 2C(s,t).\qedhere
        \]
    \end{enumerate}
\end{proof}

When working with $A(s,t)$ and $C(s,t)$, we will often use their monotonicity in $t$ without explicit reference to Lemmas \ref{ackermann-properties}i and \ref{c-properties}i.

\begin{lemma}
\label{inverse-ackermann}
    For all positive integers $s>1$ and $t$,
    \[
        A(s-1,t) \leq tC(s,t) < A(s,t+1).
    \]
\end{lemma}

\begin{proof}
    First, we will show that $tC(s,t) \geq A(s-1,t)$. For $s=2$, this is immediate, as $tC(2,t) = 2t = A(1,t)$.

    For $s>2$, we will prove the stronger statement that $C(s,t) \geq A(s-1,t)$. We will do so by double induction on $s$ and $t$, with base cases $s=3$ and $t=1$. The base case $s=3$ is immediate, as $C(3,t) = 2^t = A(2,t)$. The base case $t=1$ is also immediate, as $C(s,1) = 2 = A(s-1,1)$.

    For the inductive step, fix $s>3$, and suppose $C(s-1,t) \geq A(s-2,t)$ for all $t$. Fix $t>1$ as well, and suppose $C(s,t-1) \geq A(s-1,t-1)$. Then
    \begin{align*}
        C(s,t) &= C(s,t-1)C(s-1,C(s,t-1))\\
        &\geq C(s-1,C(s,t-1))\\
        &\geq A(s-2,A(s-1,t-1))\\
        &= A(s-1,t).
    \end{align*}
    This completes the inductive step. We may conclude that $C(s,t) \geq A(s-1,t)$ for all $s>2$, so $tC(s,t) \geq A(s-1,t)$ for all $s>1$ and all $t$.

    Second, we will show that $tC(s,t) \leq A(s,t+1)-2$. We will again use double induction on $s$ and $t$. The base cases will be $s=2,3$ and $t=1$, which we check in order.

    $s=2$:
    \[
        tC(2,t) = 2t \leq 2^{t+1}-2 = A(2,t+1)-2.
    \]
    $s=3$:
    \[
        tC(3,t) = t2^t \leq 2^{2t}-2 \leq 2^{A(3,t)}-2 = A(3,t+1)-2.
    \]
    $t=1$:
    \[
        C(s,1) = 2 = A(s,2)-2.
    \]

    For the inductive step, fix $s>3$, and assume $tC(s-1,t) \leq A(s-1,t+1)-2$ for all $t$. Fix $t>1$ as well, and assume $(t-1)C(s,t-1) \leq A(s,t)-2$.
    
    By the inductive hypothesis,
    \[
        tC(s,t-1) \leq 2(t-1)C(s,t-1) < 2A(s,t)-2.
    \]
    Using the above, Lemma \ref{c-properties}ii, and finally the inductive hypothesis again,
    \begin{align*}
        tC(s,t) &= tC(s,t-1)C(s-1,C(s,t-1))\\
        &< (2A(s,t)-2)C(s-1,A(s,t)-2)\\
        &\leq (A(s,t)-1)C(s-1,A(s,t)-1)\\
        &\leq A(s-1,A(s,t))-2\\
        &= A(s,t+1)-2.
    \end{align*}
    This completes the inductive step. We may conclude that $tC(s,t) \leq A(s,t+1)-2$ for all $s>1$ and all $t$.
\end{proof}

\subsection{Pattern-avoiding matrices}
\label{ss-matrix}

We formally introduce the notion of pattern-avoiding matrices, then proceed with our construction. As mentioned, our construction is adapted from \cite{furedi92}, and we have also consulted \cite[\S 2.3]{keszegh05} as a reference. Recall that a binary matrix is a matrix of ones and zeros.

\begin{definition}
    Let $M$ be a binary matrix. The \emph{weight} of $M$, denoted $\lvert M \rvert$, is the number of ones in $M$.
\end{definition}

\begin{definition}
    Let $M$ and $N$ be binary matrices. We say that $M$ \emph{contains} $N$ if we can obtain $N$ from $M$ by deleting a set of rows, deleting a set of columns, and changing a set of ones to zeros. We say that $M$ \emph{avoids} $N$ if we cannot obtain $N$ from $M$ in this way.
\end{definition}

Let
\[
    N=\begin{bmatrix}
        \ & 1 & \ & 1\\
        1 & \ & 1 & \
    \end{bmatrix},\qquad
    N'=\begin{bmatrix}
        1 & \ & 1\\
        1 & 1 & \
    \end{bmatrix}
\]
(zeros are left blank). We will construct arbitrarily large $n \times n$ binary matrices $M$, avoiding both $N$ and $N'$, such that $\lvert M \rvert = \Theta(n\alpha(n))$. Avoiding $N'$ is not really the goal, but it is counterintuitively easier to prove that our matrices avoid both $N$ and $N'$ than that they just avoid $N$.

For positive integers $s$ and $t$, we will define a $tC(s,t) \times tC(s,t)$ binary matrix $M(s,t)$. For any fixed $M(s,t)$, our definition will ensure that there exist $1 = c_1 < \cdots < c_{C(s,t)} \leq tC(s,t)$ such that the first $t$ rows have leading ones in column $c_1$, the next $t$ rows in column $c_2$, and so on. Following \cite{furedi92}, we will divide $M(s,t)$ into \emph{horizontal blocks} of consecutive rows, ending at rows $t, 2t, \ldots, tC(s,t)$. We will similarly divide $M(s,t)$ into \emph{vertical blocks} of consecutive columns, beginning at columns $c_1, \ldots, c_{C(s,t)}$.

\begin{construction}
    We define $M(s,t)$ recursively as follows.
    \begin{itemize}
        \item $M(1,t)$ is the $t \times t$ matrix with all ones in the first column and all zeros elsewhere.
        \item $M(s,1) = \begin{bsmallmatrix}
            1 & \ \\
            \ & 1
        \end{bsmallmatrix}$ for $s>1$.
        \item $M(s,t)$ for $s,t>1$ is obtained from $M(s,t-1)$ and $M(s-1,C(s,t-1))$ as shown in Figure \ref{mst-construction}. $H_1,\ldots,H_{C(s,t-1)}$ are the horizontal blocks of $M(s,t-1)$, which appears $C(s-1,C(s,t-1))$ many times. $V_1, \ldots, V_{C(s-1,C(s,t-1))}$ are the vertical blocks of $M(s-1,C(s,t-1))$, which appears once. $C(s,t)$ many extra ones are introduced, one per row of $M(s-1,C(s,t-1))$.
    \end{itemize}
\end{construction}

\begin{figure}[p]
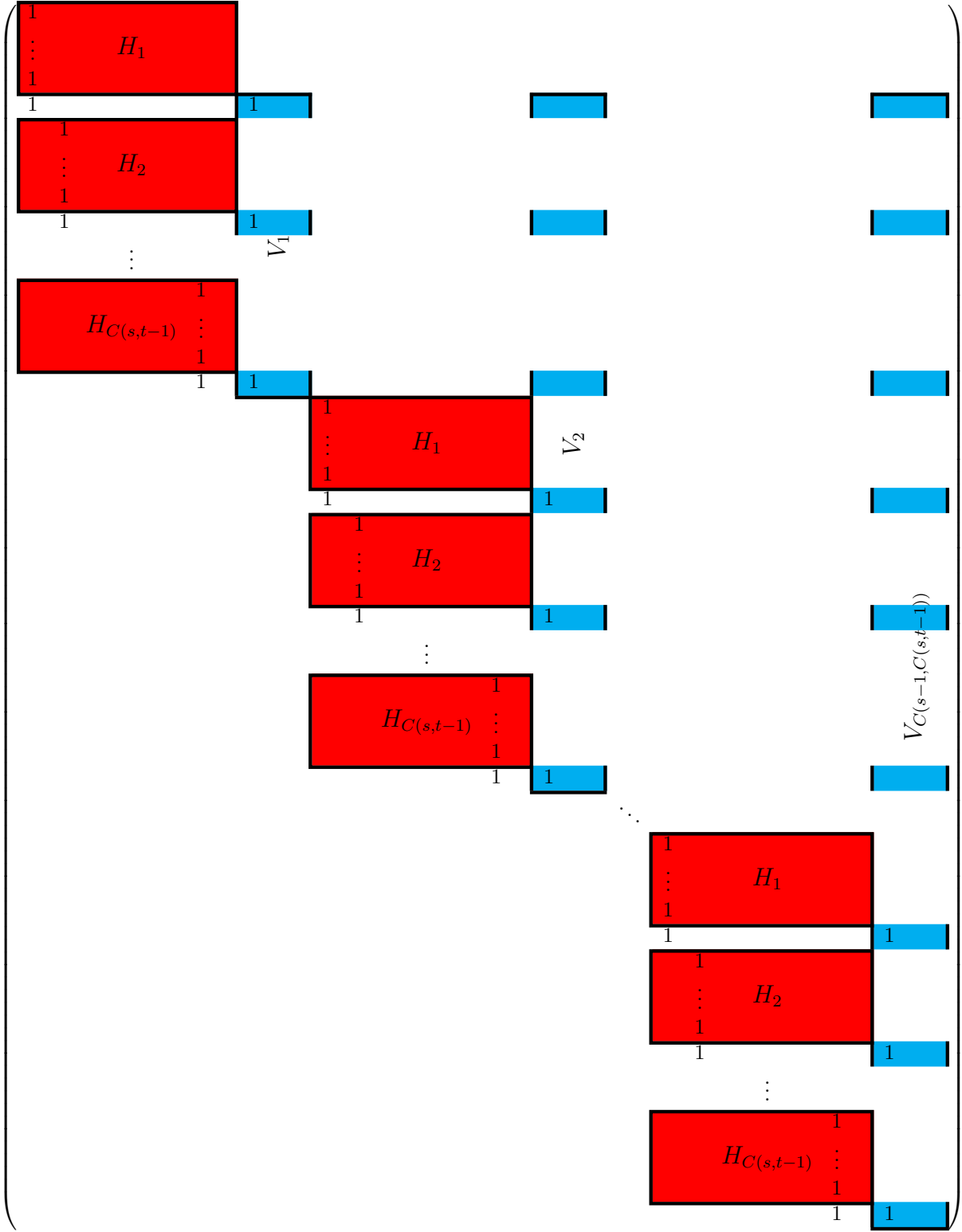

    \[
    \begin{pNiceMatrix}[margin,rules/width=1.5pt]
        \Block[fill=red,draw]{3-9}<\large>{H_1} 1\\
        \vdots\\
        1\\
        1 &&&&&&&&& \Block[fill=cyan,borders={top,left,right}]{1-3}{} \Block{10-3}<\large\rotate>{V_1} 1 &&&&&&&&&&&& \Block[fill=cyan,borders={top,left,right}]{1-3}{} \Block{23-3}<\large\rotate>{V_2} &&&&&&&&&&&&& \Block[fill=cyan,borders={top,left,right}]{1-3}{} \Block{37-3}<\large\rotate>{V_{C(s-1,C(s,t-1))}}\\
        \Block[fill=red,draw]{3-9}<\large>{H_2}& 1\\
       & \vdots\\
       & 1\\
       & 1 &&&&&&&& \Block[fill=cyan,borders={left,right}]{1-3}{} 1 &&&&&&&&&&&& \Block[fill=cyan,borders={left,right}]{1-3}{}&&&&&&&&&&&&& \Block[fill=cyan,borders={left,right}]{1-3}{}\\
        \Block{1-9}{\vdots}\\
        \Block[fill=red,draw]{3-9}<\large>{H_{C(s,t-1)}}&&&&&&& 1\\
       &&&&&&& \vdots\\
       &&&&&&& 1\\
       &&&&&&& 1 && \Block[fill=cyan,borders={bottom,left,right}]{1-3}{} 1 &&&&&&&&&&&& \Block[fill=cyan,borders={left,right}]{1-3}{}&&&&&&&&&&&&& \Block[fill=cyan,borders={left,right}]{1-3}{}\\
       &&&&&&&&&&&& \Block[fill=red,draw]{3-9}<\large>{H_1} 1\\
       &&&&&&&&&&&& \vdots\\
       &&&&&&&&&&&& 1\\
       &&&&&&&&&&&& 1 &&&&&&&&& \Block[fill=cyan,borders={left,right}]{1-3}{} 1 &&&&&&&&&&&&& \Block[fill=cyan,borders={left,right}]{1-3}{}\\
       &&&&&&&&&&&& \Block[fill=red,draw]{3-9}<\large>{H_2}& 1\\
       &&&&&&&&&&&&& \vdots\\
       &&&&&&&&&&&&& 1\\
       &&&&&&&&&&&&& 1 &&&&&&&& \Block[fill=cyan,borders={left,right}]{1-3}{} 1 &&&&&&&&&&&&& \Block[fill=cyan,borders={left,right}]{1-3}{}\\
       &&&&&&&&&&&& \Block{1-9}{\vdots}\\
       &&&&&&&&&&&& \Block[fill=red,draw]{3-9}<\large>{H_{C(s,t-1)}}&&&&&&& 1\\
       &&&&&&&&&&&&&&&&&&& \vdots\\
       &&&&&&&&&&&&&&&&&&& 1\\
       &&&&&&&&&&&&&&&&&&& 1 && \Block[fill=cyan,borders={bottom,left,right}]{1-3}{} 1 &&&&&&&&&&&&& \Block[fill=cyan,borders={left,right}]{1-3}{}\\
        &&&&&&&&&&&&&&&&&&&&&&&&\ddots\\
       &&&&&&&&&&&&&&&&&&&&&&&&& \Block[fill=red,draw]{3-9}<\large>{H_1} 1\\
       &&&&&&&&&&&&&&&&&&&&&&&&& \vdots\\
       &&&&&&&&&&&&&&&&&&&&&&&&& 1\\
       &&&&&&&&&&&&&&&&&&&&&&&&& 1 &&&&&&&&& \Block[fill=cyan,borders={left,right}]{1-3}{} 1\\
       &&&&&&&&&&&&&&&&&&&&&&&&& \Block[fill=red,draw]{3-9}<\large>{H_2}& 1\\
       &&&&&&&&&&&&&&&&&&&&&&&&&& \vdots\\
       &&&&&&&&&&&&&&&&&&&&&&&&&& 1\\
       &&&&&&&&&&&&&&&&&&&&&&&&&& 1 &&&&&&&& \Block[fill=cyan,borders={left,right}]{1-3}{} 1\\
       &&&&&&&&&&&&&&&&&&&&&&&&& \Block{1-9}{\vdots}\\
       &&&&&&&&&&&&&&&&&&&&&&&&& \Block[fill=red,draw]{3-9}<\large>{H_{C(s,t-1)}}&&&&&&& 1\\
       &&&&&&&&&&&&&&&&&&&&&&&&&&&&&&&& \vdots\\
       &&&&&&&&&&&&&&&&&&&&&&&&&&&&&&&& 1\\
       &&&&&&&&&&&&&&&&&&&&&&&&&&&&&&&& 1 && \Block[fill=cyan,borders={bottom,left,right}]{1-3}{} 1 &&
    \end{pNiceMatrix}
    \]
    \caption{Construction of $M(s,t)$ from $C(s-1,C(s,t-1))$ many copies of $M(s,t-1)$, one copy of $M(s-1,C(s,t-1))$, and $C(s,t)$ many extra ones. $H_1, \ldots, H_{C(s,t-1)}$ are the horizontal blocks of $M(s,t-1)$; $V_1, \ldots, V_{C(s-1,C(s,t-1))}$ are the vertical blocks of $M(s-1,C(s,t-1))$.}
    \label{mst-construction}
\end{figure}

We refer the reader to \cite[Definition 7.2]{furedi92} for a narrative description of the recursive step above; our recursive step (Figure \ref{mst-construction}) is identical. However, we use different base cases $M(1,t)$ and $M(s,1)$ to ensure that each has the claimed number of horizontal and vertical blocks. Small examples of $M(s,t)$ are given in Table \ref{matrices}.

\begin{table}[ht!]
\begin{align*}
M(1,1)&=\begin{bsmallmatrix}
    1
\end{bsmallmatrix},
&
M(1,2)&=\begin{bsmallmatrix}
    1 & \ \\
    1 & \ 
\end{bsmallmatrix},
&
M(1,3)&=\begin{bsmallmatrix}
    1 & \ & \ \\
    1 & \ & \ \\
    1 & \ & \
\end{bsmallmatrix},
&
M(1,4)&=\begin{bsmallmatrix}
    1 & \ & \ & \ \\
    1 & \ & \ & \ \\
    1 & \ & \ & \ \\
    1 & \ & \ & \
\end{bsmallmatrix},&\ldots,\\
M(2,1)&=\begin{bsmallmatrix}
    1 & \ \\
    \ & 1
\end{bsmallmatrix},
&
M(2,2)&=\begin{bsmallmatrix}
    \color{red}\textbf{1} & \ & \ & \ \\
    1 & \ & \color{cyan}\underline{1} & \ \\
    \ & \color{red}\textbf{1} & \ & \ \\
    \ & 1 & \color{cyan}\underline{1} & \
\end{bsmallmatrix},
&
M(2,3)&=\begin{bsmallmatrix}
    \color{red}\textbf{1} & \ & \ & \ & \ & \ \\
    \color{red}\textbf{1} & \ & \color{red}\textbf{1} & \ & \ & \ \\
    1 & \ & \ & \ & \color{cyan}\underline{1} & \ \\
    \ & \color{red}\textbf{1} & \ & \ & \ & \ \\
    \ & \color{red}\textbf{1} & \color{red}\textbf{1} & \ & \ & \ \\
    \ & 1 & \ & \ & \color{cyan}\underline{1} & \
\end{bsmallmatrix},
&
M(2,4)&=\begin{bsmallmatrix}
    \color{red}\textbf{1} & \ & \ & \ & \ & \ & \ & \ \\
    \color{red}\textbf{1} & \ & \color{red}\textbf{1} & \ & \ & \ & \ & \ \\
    \color{red}\textbf{1} & \ & \ & \ & \color{red}\textbf{1} & \ & \ & \ \\
    1 & \ & \ & \ & \ & \ & \color{cyan}\underline{1} & \ \\
    \ & \color{red}\textbf{1} & \ & \ & \ & \ & \ & \ \\
    \ & \color{red}\textbf{1} & \color{red}\textbf{1} & \ & \ & \ & \ & \ \\
    \ & \color{red}\textbf{1} & \ & \ & \color{red}\textbf{1} & \ & \ & \ \\
    \ & 1 & \ & \ & \ & \ & \color{cyan}\underline{1} & \
\end{bsmallmatrix},&\ldots,\\
M(3,1)&=\begin{bsmallmatrix}
    1 & \\
   \ & 1
\end{bsmallmatrix},
&
M(3,2)&=\begin{bsmallmatrix}
    \color{red}\textbf{1} & \ & \ & \ & \ & \ & \ & \ \\
    1 & \ & \color{cyan}\underline{1} & \ & \ & \ & \ & \ \\
    \ & \color{red}\textbf{1} & \ & \ & \ & \ & \ & \ \\
    \ & 1 & \color{cyan}\underline{1} & \ & \ & \ & \color{cyan}\underline{1} & \ \\
    \ & \ & \ & \color{red}\textbf{1} & \ & \ & \ & \ \\
    \ & \ & \ & 1 & \ & \color{cyan}\underline{1} & \ & \ \\
    \ & \ & \ & \ & \color{red}\textbf{1} & \ & \ & \ \\
    \ & \ & \ & \ & 1 & \color{cyan}\underline{1} & \color{cyan}\underline{1} & \
\end{bsmallmatrix},
&
M(3,3)&=\mathrlap{\begin{bsmallmatrix}
    \color{red}\textbf{1} & \ & \ & \ & \ & \ & \ & \ & \ & \ & \ & \ & \ & \ & \ & \ & \ & \ & \ & \ & \ & \ & \ \\
    \color{red}\textbf{1} & \ & \color{red}\textbf{1} & \ & \ & \ & \ & \ & \ & \ & \ & \ & \ & \ & \ & \ & \ & \ & \ & \ & \ & \ & \ \\
    1 & \ & \ & \ & \ & \ & \ & \ & \color{cyan}\underline{1} & \ & \ & \ & \ & \ & \ & \ & \ & \ & \ & \ & \ & \ & \ \\
    \ & \color{red}\textbf{1} & \ & \ & \ & \ & \ & \ & \ & \ & \ & \ & \ & \ & \ & \ & \ & \ & \ & \ & \ & \ & \ \\
    \ & \color{red}\textbf{1} & \color{red}\textbf{1} & \ & \ & \ & \color{red}\textbf{1} & \ & \ & \ & \ & \ & \ & \ & \ & \ & \ & \ & \ & \ & \ & \ & \ \\
    \ & 1 & \ & \ & \ & \ & \ & \ & \color{cyan}\underline{1} & \ & \ & \ & \ & \ & \ & \ & \ & \ & \color{cyan}\underline{1} & \ & \ & \ & \ & \ \\
    \ & \ & \ & \color{red}\textbf{1} & \ & \ & \ & \ & \ & \ & \ & \ & \ & \ & \ & \ & \ & \ & \ & \ & \ & \ \\
    \ & \ & \ & \color{red}\textbf{1} & \ & \color{red}\textbf{1} & \ & \ & \ & \ & \ & \ & \ & \ & \ & \ & \ & \ & \ & \ & \ & \ & \ \\
    \ & \ & \ & 1 & \ & \ & \ & \ & \color{cyan}\underline{1} & \ & \ & \ & \ & \ & \ & \ & \ & \ & \ & \ & \color{cyan}\underline{1} & \ & \ & \ \\
    \ & \ & \ & \ & \color{red}\textbf{1} & \ & \ & \ & \ & \ & \ & \ & \ & \ & \ & \ & \ & \ & \ & \ & \ & \ \\
    \ & \ & \ & \ & \color{red}\textbf{1} & \color{red}\textbf{1} & \color{red}\textbf{1} & \ & \ & \ & \ & \ & \ & \ & \ & \ & \ & \ & \ & \ & \ & \ \\
    \ & \ & \ & \ & 1 & \ & \ & \ & \color{cyan}\underline{1} & \ & \ & \ & \ & \ & \ & \ & \ & \ & \ & \ & \ & \ & \color{cyan}\underline{1} & \ \\
    \ & \ & \ & \ & \ & \ & \ & \ & \ & \color{red}\textbf{1} & \ & \ & \ & \ & \ & \ & \ & \ & \ & \ & \ & \ & \ \\
    \ & \ & \ & \ & \ & \ & \ & \ & \ & \color{red}\textbf{1} & \ & \color{red}\textbf{1} & \ & \ & \ & \ & \ & \ & \ & \ & \ & \ & \ \\
    \ & \ & \ & \ & \ & \ & \ & \ & \ & 1 & \ & \ & \ & \ & \ & \ & \ & \color{cyan}\underline{1} & \ & \ & \ & \ & \ & \ \\
    \ & \ & \ & \ & \ & \ & \ & \ & \ & \ & \color{red}\textbf{1} & \ & \ & \ & \ & \ & \ & \ & \ & \ & \ & \ & \ \\
    \ & \ & \ & \ & \ & \ & \ & \ & \ & \ & \color{red}\textbf{1} & \color{red}\textbf{1} & \ & \ & \ & \color{red}\textbf{1} & \ & \ & \ & \ & \ & \ & \ \\
    \ & \ & \ & \ & \ & \ & \ & \ & \ & \ & 1 & \ & \ & \ & \ & \ & \ & \color{cyan}\underline{1} & \color{cyan}\underline{1} & \ & \ & \ & \ & \ \\
    \ & \ & \ & \ & \ & \ & \ & \ & \ & \ & \ & \ & \color{red}\textbf{1} & \ & \ & \ & \ & \ & \ & \ & \ & \ & \ \\
    \ & \ & \ & \ & \ & \ & \ & \ & \ & \ & \ & \ & \color{red}\textbf{1} & \ & \color{red}\textbf{1} & \ & \ & \ & \ & \ & \ & \ & \ \\
    \ & \ & \ & \ & \ & \ & \ & \ & \ & \ & \ & \ & 1 & \ & \ & \ & \ & \color{cyan}\underline{1} & \ & \ & \color{cyan}\underline{1} & \ & \ & \ \\
    \ & \ & \ & \ & \ & \ & \ & \ & \ & \ & \ & \ & \ & \color{red}\textbf{1} & \ & \ & \ & \ & \ & \ & \ & \ & \ \\
    \ & \ & \ & \ & \ & \ & \ & \ & \ & \ & \ & \ & \ & \color{red}\textbf{1} & \color{red}\textbf{1} & \color{red}\textbf{1} & \ & \ & \ & \ & \ & \ & \ \\
    \ & \ & \ & \ & \ & \ & \ & \ & \ & \ & \ & \ & \ & 1 & \ & \ & \ & \color{cyan}\underline{1} & \ & \ & \ & \ & \color{cyan}\underline{1} & \
\end{bsmallmatrix},} &&&\ldots.
\end{align*}
\caption{Selected matrices $M(s,t)$, colored as in the proof of Lemma \ref{avoiding}.}
\label{matrices}
\end{table}

\begin{lemma}
\label{avoiding}
    $M(s,t)$ avoids both $N$ and $N'$ for all positive integers $s$ and $t$.
\end{lemma}

\begin{proof}
    We argue by double induction on $s$ and $t$. For the base cases, observe that each matrix $M(1,t)$ or $M(s,1)$ avoids both $N$ and $N'$.

    For the inductive step, fix $s,t>1$, and suppose $M(s,t-1)$ and $M(s-1,C(s,t-1))$ each avoid $N,N'$. In $M(s,t)$, color each ``1" red, blue, or black: red if inherited from a copy of $M(s,t-1)$, blue if inherited from $M(s-1,C(s,t-1))$, or black otherwise. Our construction of $M(s,t)$ implies (a)-(d) below, while the inductive hypothesis implies (e). We \textbf{bold} red ones and \underline{underline} blue ones to assist colorblind readers.
    \begin{enumerate}[label=(\alph*)]
        \item A red ``1" cannot be in the same row as a blue or black ``1".
        \item A blue ``1" cannot be in the same column as a red or black ``1".
        \item Given a black ``1", any other ``1" must be in a higher row or a further-right column.
        \item $M(s,t)$ avoids the colored matrices
        \[
            \begin{bmatrix}
                \ & \color{red}\textbf{1}\\
                \color{cyan}\underline{1} & \
            \end{bmatrix},\qquad
            \begin{bmatrix}
                \ & \color{cyan}\underline{1} & \ \\
                \color{red}\textbf{1} & \ & \color{red}\textbf{1}
            \end{bmatrix}.
        \]
        \item $M(s,t)$ avoids the colored matrices
        \[
            \begin{bmatrix}
                \ & \color{red}\textbf{1} & \ & \color{red}\textbf{1}\\
                \color{red}\textbf{1} & \ & \color{red}\textbf{1} & \
            \end{bmatrix},\qquad
            \begin{bmatrix}
                \color{red}\textbf{1} & \ & \color{red}\textbf{1}\\
                \color{red}\textbf{1} & \color{red}\textbf{1} & \
            \end{bmatrix},\qquad
            \begin{bmatrix}
                \ & \color{cyan}\underline{1} & \ & \color{cyan}\underline{1}\\
                \color{cyan}\underline{1} & \ & \color{cyan}\underline{1} & \
            \end{bmatrix},\qquad
            \begin{bmatrix}
                \color{cyan}\underline{1} & \ & \color{cyan}\underline{1}\\
                \color{cyan}\underline{1} & \color{cyan}\underline{1} & \
            \end{bmatrix}.
        \]
    \end{enumerate}
    No coloring of $N'$ is consistent with (a)-(e), so $M(s,t)$ must avoid $N'$.
    
    It remains to prove that $M(s,t)$ avoids $N$. By way of contradiction, suppose $M(s,t)$ contains $N$. Let $1 \leq r_1 < r_2 \leq tC(s,t)$ be the rows and $1 \leq k_1 < k_2 < k_3 < k_4 \leq tC(s,t)$ the columns of $M(s,t)$ in which $N$ appears. Then the unique coloring of $N$ consistent with (a)-(e) is
    \[
        \begin{bNiceMatrix}[first-row,first-col]
            & k_1 & k_2 & k_3 & k_4\\
            r_1 & \ & \color{cyan}\underline{1} & \ & \color{cyan}\underline{1}\\
            r_2 & 1 & \ & \color{cyan}\underline{1} & \
        \end{bNiceMatrix}.
    \]
    
    Let $V_i$ be the first vertical block of $M(s-1,C(s,t-1))$ that appears in $M(s,t)$ to the right of column $k_1$. Let $k>k_1$ be the column containing the first blue ``1" in row $r_2$. By construction, $k$ is the first column of $V_i$. It follows that $k \leq k_2$; otherwise, the blue ``1" in position $(r_1,k_2)$ must have come from a vertical block $V_j$ of $M(s-1,C(s,t-1))$ with $j<i$, contradicting the minimality of $i$.

    $M(s,t)$ therefore contains one of the colored matrices
    \[
        \begin{bNiceMatrix}[first-row,first-col]
            & k_1 & & k_2 & k_3 & k_4\\
            r_1 & \ & \ & \color{cyan}\underline{1} & \ & \color{cyan}\underline{1}\\
            r_2 & 1 & \color{cyan}\underline{1} & \ & \color{cyan}\underline{1} & \
        \end{bNiceMatrix},\qquad
        \begin{bNiceMatrix}[first-row,first-col]
            & k_1 & k_2 & k_3 & k_4\\
            r_1 & \ & \color{cyan}\underline{1} & \ & \color{cyan}\underline{1}\\
            r_2 & 1 & \color{cyan}\underline{1} & \color{cyan}\underline{1} & \
        \end{bNiceMatrix}.
    \]
    This contradicts (e). We may conclude that $M(s,t)$ avoids $N$, and the inductive step is complete.
\end{proof}

For binary matrices $M$, recall that $\lvert M \rvert$ denotes the weight of $M$. By construction, $\lvert M(s,t) \rvert$ satisfies the following for all $s$ and $t$.
\begin{itemize}
    \item $\lvert M(1,t) \rvert = t$ for all $t$.
    \item $\lvert M(s,1) \rvert = 2$ for all $s>1$.
    \item $\lvert M(s,t) \rvert = C(s,t) + C(s-1,C(s,t-1))\lvert M(s,t-1)\rvert + \lvert M(s-1,C(s,t-1)) \rvert$ for $s,t>1$.
\end{itemize}

For $s,t>1$, the recursive formula for $\lvert M(s,t) \rvert$ may be rearranged to obtain
\[
    \frac{\vert M(s,t) \rvert}{C(s,t)} - \frac{\lvert M(s,t-1) \rvert}{C(s,t-1)} = 1 + \frac{\lvert M(s-1,C(s,t-1)) \rvert}{C(s,t-1)C(s-1,C(s,t-1))}.
\]
If we substitute the dummy variable $u$ for $t-1$ in the above, then sum the resulting equation from $u=1$ to $u=t-1$, we find
\begin{alignat}{3}
    &&\frac{\lvert M(s,t) \rvert}{C(s,t)} - \frac{\lvert M(s,1) \rvert}{C(s,1)} &= (t-1) + \sum_{u=1}^{t-1} \frac{\lvert M(s-1,C(s,u)) \rvert}{C(s,u)C(s-1,C(s,u))}\notag\\
    \mathllap{\Rightarrow\quad}&& \frac{\lvert M(s,t) \rvert}{C(s,t)} &= t + \sum_{u=1}^{t-1} \frac{\lvert M(s-1,C(s,u)) \rvert}{C(s,u)C(s-1,C(s,u))}\notag\\
    \mathllap{\Rightarrow\quad}&& \frac{\lvert M(s,t) \rvert}{tC(s,t)} &= 1 + \frac{1}{t}\sum_{u=1}^{t-1} \frac{\lvert M(s-1,C(s,u)) \rvert}{C(s,u)C(s-1,C(s,u))}.\label{density-recursion}
\end{alignat}

\begin{lemma}
\label{weight}
    For all positive integers $s$ and $t$ with $t \geq s$,
    \[
        \frac{s}{3} < \frac{\lvert M(s,t) \rvert}{tC(s,t)} \leq s.
    \]
\end{lemma}

\begin{proof}
    First, we will prove that $\lvert M(s,t) \rvert/(tC(s,t)) > s/3$ for $t \geq s$. We argue by induction on $s$. The base case $s=1$ is immediate, as $\lvert M(1,t) \rvert/(tC(1,t)) = 1 > 1/3$ for all $t$.

    For the inductive step, fix $s>1$, and suppose that $\lvert M(s-1,t) \rvert/(tC(s-1,t)) > (s-1)/3$ for all $t\geq s-1$. For all $u>1$, we can observe that $C(s,u) \geq C(s,2) = 2^{s-1} > s-1$. Thus, by the inductive hypothesis, for all $u>1$,
    \[
        \frac{\lvert M(s-1,C(s,u)) \rvert}{C(s,u)C(s-1,C(s,u))} > \frac{s-1}{3}.
    \]
    Hence, by (\ref{density-recursion}), for all $t\geq s$,
    \[
        \frac{\lvert M(s,t) \rvert}{tC(s,t)} \geq 1+\frac{(s-1)(t-2)}{3t} > \frac{s}{3}.
    \]
    This completes the inductive step. We may conclude that $\lvert M(s,t) \rvert/(tC(s,t)) > s/3$ for $t \geq s$.

    Second, we will prove that $\lvert M(s,t) \rvert/(tC(s,t)) \leq s$ for all $s$ and $t$. We argue by induction on $s$. The base case $s=1$ is again immediate, as $\lvert M(1,t) \rvert/(tC(1,t)) = 1$ for all $t$.

    For the inductive step, fix $s>1$, and suppose that $\lvert M(s-1,t) \rvert/(tC(s-1,t))\leq s-1$ for all $t$. If $t=1$, we have $\lvert M(s,1) \rvert/C(s,1)=1\leq s$, so we are done. Suppose $t>1$. By the inductive hypothesis, for all $u$,
    \[
        \frac{\lvert M(s-1,C(s,u)) \rvert}{C(s,u)C(s-1,C(s,u))} \leq s-1.
    \]
    Thus, by (\ref{density-recursion}),
    \[
        \frac{\lvert M(s,t) \rvert}{tC(s,t)} \leq 1 + \frac{(s-1)(t-1)}{t} < s.
    \]
    This completes the inductive step. We may conclude that $\lvert M(s,t) \rvert/(tC(s,t)) \leq s$ for all $s$ and $t$. Hence, $s/3 < \lvert M(s,t) \rvert/(tC(s,t)) \leq s$ for all $s$ and $t\geq s$.
\end{proof}

\begin{theorem}[F\"uredi and Hajnal]
\label{ds-matrices}
    For arbitrarily large $n$, there exist $n \times n$ binary matrices $M$ avoiding $N$ such that $\lvert M \rvert = \Theta(n\alpha(n))$.
\end{theorem}

\begin{proof}
    Let $n=sC(s,s)$ and $M=M(s,s)$. Then $M$ is an $n \times n$ binary matrix avoiding $N$ (Lemma \ref{avoiding}). Furthermore, by Lemmas \ref{inverse-ackermann} and \ref{weight}, we have $\lvert M \rvert = \Theta(s^2C(s,s)) = \Theta(n\alpha(n))$.
\end{proof}

\begin{remark}
    The growth rate $\Theta(n\alpha(n))$ in Theorem \ref{ds-matrices} is the highest possible (see \cite[Theorem 8.1]{furedi92}).
\end{remark}


\subsection{Polytopal 3-complexes}
\label{ss-moment-curve}

To conclude this section, we use our pattern-avoiding matrices to construct polytopal 3-complexes with unbounded complexity. Recall that complexity is defined as $(f_{03}-20)/(f_0+f_3-10)$.

\begin{definition}
    The \emph{moment curve} $x:\mathbb{R} \to \mathbb{R}^d$ is the parametric curve defined by
    \[
        x(t) = (t, t^2 \ldots, t^d).
    \]
\end{definition}

\begin{remark}
    It is well known that any finite set of points on the moment curve is in general position (i.e. a subset of size $k$ has affine hull of dimension $\min\{k-1,d\}$).
\end{remark}

The most famous combinatorial property of the moment curve is \emph{Gale's evenness criterion}, which characterizes the facets of any polytope with vertices on the moment curve. We will use a related property derived from Gale's criterion and Radon's theorem by Lee and Nevo \cite[Propositions 2.2\&2.4]{lee24}.

\begin{proposition}[Lee and Nevo]
    Let $A,B \subset \mathbb{R}$ be finite sets with $\lvert A\cap B \rvert \leq d$, and let $x:\mathbb{R}\to\mathbb{R}^d$ be the moment curve. Then $\operatorname{conv}x(A) \cap \operatorname{conv} x(B) \supsetneq \operatorname{conv}x(A\cap B)$ if and only if there exist $t_1 < \cdots < t_{d+2}$ such that either
    \begin{itemize}
        \item $\{t_i \mid i \text{~odd}\} \subseteq A$ and $\{t_i \mid i \text{~even}\} \subseteq B$, or
        \item $\{t_i \mid i \text{~odd}\} \subseteq B$ and $\{t_i \mid i \text{~even}\} \subseteq A$.
    \end{itemize}
\end{proposition}

\begin{corollary}
\label{incidence-matrix}
    Let $M$ be an $m\times n$ binary matrix with pairwise distinct rows, and let $x:\mathbb{R}\to\mathbb{R}^3$ be the moment curve. For $i=1,\ldots,m$, let $P_i=\operatorname{conv}\{x(j) \mid M_{ij}=1\}$. Then the collection of polytopes $P_1, \ldots, P_m$ and their faces is a polytopal 3-complex if and only if $M$ avoids the matrices
    \[
        \begin{bmatrix}
            1 & 1 & 1 & 1\\
            1 & 1 & 1 & 1
        \end{bmatrix},\qquad
        \begin{bmatrix}
            1 & \ & 1 & \ & 1\\
            \ & 1 & \ & 1 & \
        \end{bmatrix},\qquad
        \begin{bmatrix}
            \ & 1 & \ & 1 & \ \\
            1 & \ & 1 & \ & 1
        \end{bmatrix}.
    \]
\end{corollary}

\begin{theorem}
    For arbitrarily large $n$, there exist polytopal 3-complexes $Q$ such that $f_0(Q),f_3(Q) \leq n$ and $f_{03}(Q) = \Theta(n\alpha(n))$.
\end{theorem}

\begin{proof}
    Let $M$ be an $n\times n$ matrix as in Theorem \ref{ds-matrices}. Delete any row of $M$ containing fewer than four ones, and let $m\leq n$ be the number of remaining rows. Since we have deleted at most $3n$ many ones, we still have $\lvert M \rvert = \Theta(n\alpha(n))$.

    Let $P_i=\operatorname{conv}\{ x(j) \mid M_{ij}=1 \}$ for $i=1,\ldots,m$. Then each $P_i$ is a 3-polytope, since it is the convex hull of at least four points in general position. Let $Q$ be the collection of 3-polytopes $P_1, \ldots, P_m$ and their faces. By Corollary \ref{incidence-matrix}, $Q$ is a polytopal 3-complex.

    We have $f_3(Q)=m$, and $f_0(Q)$ is the number of nonempty columns of $M$. Thus, $f_0(Q),f_3(Q) \leq n$. Finally, by construction, we have $f_{03}(Q)=\lvert M \rvert = \Theta(n\alpha(n))$.
\end{proof}


\section{Shellable spheres}
\label{s-main-section}

This section is dedicated to our main construction: a family of shellable, dual shellable 3-spheres with unbounded fatness. We obtain these spheres by gluing together balls. Where we previously constructed a matrix from two smaller matrices, entering ones where columns of the first matrix intersect rows of the second, we now construct a ball from two smaller balls, creating incidences between vertices of the first ball and facets of the second.

We begin by introducing new recursive functions that will describe combinatorial properties of our spheres (\S\ref{ss-functions}). We then give our main construction (\S\ref{ss-construction}). Next, we prove that our spheres are shellable and dual shellable (\S\ref{ss-shellable}) and give asymptotic bounds on their face numbers (\S\ref{ss-asymptotics}). We use these bounds to prove our main result: the existence of arbitrarily fat, shellable, dual shellable 3-spheres (\S\ref{ss-main-result}). Finally, we conjecture that a subfamily of our spheres can be realized as arbitrarily fat 4-polytopes (\S\ref{ss-questions}).

\subsection{Auxiliary functions}
\label{ss-functions}

We aim to construct 3-spheres using a similar double recursion as in Section \ref{s-binary-matrices}. In place of $C(s,t)$, we will use two new functions $K(s,t)$ and $K'(s,t)$.

\begin{definition}
\label{def-k}
    For positive integers $s$ and $t$, the functions $K(s,t)$ and $K'(s,t)$ are defined recursively as follows.
    \begin{itemize}
        \item $K(1,t)=2,\ K'(1,t)=5$ for all $t$.
        \item $K(s,1)=2^s,\ K'(s,1)=2^s+3$ for all $s$.
        \item For $s,t>1$,
            \begin{align*}
                K(s,t) &= K(s,t-1)K(s-1,K'(s,t-1)),\\
                K'(s,t) &= K'(s,t-1)K(s-1,K'(s,t-1)) + K'(s-1,K'(s,t-1)).
            \end{align*}
    \end{itemize}
\end{definition}

Values of $K(s,t)$ and $K'(s,t)$ for small $s$ and $t$ are given in Table \ref{K-table}. We note that there is a more concise way to define $K(s,t)$ and $K'(s,t)$ for $t \geq 0$ by setting $K(s,0)=2$ and $K'(s,0)=1$ for $s>1$, but we avoid this to remain consistent in our use of positive indices.

\begin{table}[t]
\[
    \begin{NiceArray}[t,first-row]{|c|lll|}
        \multicolumn{4}{c}{K(s,t)}\\
        \hline
        \diagbox{s}{t}& 1 & 2 & 3\\
        \hline
        1 & 2 & 2 & 2\\
        2 & 4 & 8 & 16\\
        3 & 8 & 2^{15} & 2^{7\cdot2^{13}+11}\\
        \hline
    \end{NiceArray}
    \qquad
    \begin{NiceArray}[t,first-row]{|c|lll|}
        \multicolumn{4}{c}{K'(s,t)}\\
        \hline
        \diagbox{s}{t}& 1 & 2 & 3\\
        \hline
        1 & 5 & 5 & 5\\
        2 & 7 & 19 & 43\\
        3 & 11 & 7\cdot2^{13}-5 & (7\cdot2^{12}-1)2^{7\cdot2^{13}-3}-5\\
        \hline
    \end{NiceArray}
\]
    \caption{Values of $K(s,t)$ and $K'(s,t)$ for $s,t < 4$.}
    \label{K-table}
\end{table}

\begin{observation}
\label{k2t}
    It is easy to check that for all positive integers $t$,
    \[
        K(2,t)=2^{t+1},\qquad K'(2,t)=3\cdot 2^{t+1}-5.
    \]
\end{observation}

\begin{lemma}
\label{K-properties}
    Let $s$ and $t$ be positive integers.
    \begin{enumerate}[label=\roman*.]
        \item $K(s,t+1) \geq K(s,t)$.
        \item $K'(s,t+1)\geq K'(s,t)$.
        \item $K'(s,t+1)>2K'(s,t)$ for $s>1$.
        \item $K'(s,t)>K(s,t)$.
        \item $K(s,t)\geq2^{st-t+1}$.
    \end{enumerate}
\end{lemma}

\begin{proof}
    We prove i-v in order.
    \begin{enumerate}[wide,label=\roman*.]
        \item By definition, $K(1,t+1)=2=K(1,t)$ for all $t$. If $s>1$, then
        \[
            K(s,t+1) = K(s,t)K(s-1,K'(s,t)) \geq K(s,t).
        \]
        \item By definition, $K'(1,t+1)=5=K'(1,t)$ for all $t$. If $s>1$, then
        \[
            K'(s,t+1) = K'(s,t)K(s-1,K'(s,t))+K'(s-1,K'(s,t)) > K'(s,t).
        \]
        \item Let $s>1$. It follows from i that $K(s-1,t) \geq K(s-1,1) \geq 2$ for all $t$. Thus,
        \[
            K'(s,t+1) = K'(s,t)K(s-1,K'(s,t))+K'(s-1,K'(s,t)) > 2K'(s,t).
        \]
        \item We know $K'(1,t)=5>2=K(1,t)$ for all $t$, and $K'(s,1)=2^s+3>2^s=K(s,1)$ for all $s$. The cases $s,t>1$ follow by induction:
        \[
            K'(s,t) = K'(s,t-1)K(s-1,K'(s,t-1))+K'(s-1,K'(s,t-1)) > K(s,t-1)K(s-1,K'(s,t-1)) = K(s,t).
        \]
        \item For $s=1$ and any $t$, we have $K(1,t)=2=2^{st-t+1}$. For $t=1$ and any $s$, we have $K(s,1)=2^s=2^{st-t+1}$. The cases $s,t>1$ follow by induction:
        \[
            K(s,t) = K(s,t-1)K(s-1,K'(s,t-1)) \geq 2^{s(t-1)-t+2}2^{s-1} = 2^{st-t+1}.\qedhere
        \]
    \end{enumerate}
\end{proof}

For convenience, we will generally use Lemma \ref{K-properties}i-ii without explicit reference.

\begin{lemma}
\label{K-quotient}
    For all positive integers $s$ and $t$,
    \[
        1 < \frac{K'(s,t)}{K(s,t)} < 1+2^{3-s}.
    \]
\end{lemma}

\begin{proof}
    The inequality $K'(s,t)/K(s,t)>1$ follows immediately from Lemma \ref{K-properties}iv. It remains to prove that $K'(s,t)/K(s,t)<1+2^{3-s}$.

    For $s,t>1$, the recursive formulas for $K(s,t)$ and $K'(s,t)$ imply
    \[
        \frac{K'(s,t)}{K(s,t)} - \frac{K'(s,t-1)}{K(s,t-1)} = \frac{1}{K(s,t-1)} \cdot \frac{K'(s-1,K'(s,t-1))}{K(s-1,K'(s,t-1))}.
    \]
    Substituting the dummy variable $u$ for $t-1$ and summing from $u=1$ to $u=t-1$, we find
    \begin{alignat}{3}
        &&\frac{K'(s,t)}{K(s,t)} - \frac{K'(s,1)}{K(s,1)} &= \sum_{u=1}^{t-1} \frac{1}{K(s,u)} \cdot \frac{K'(s-1,K'(s,u))}{K(s-1,K'(s,u))}\notag\\
        \mathllap{\Rightarrow\quad}&& \frac{K'(s,t)}{K(s,t)} &= 1+3\cdot2^{-s}+\sum_{u=1}^{t-1} \frac{1}{K(s,u)} \cdot \frac{K'(s-1,K'(s,u))}{K(s-1,K'(s,u))}.\label{K-quotient-formula}
    \end{alignat}

    We proceed by induction on $s$. Our base cases are $s=1,2$. For $s=1$, we have $K'(1,t)/K(1,t)\allowbreak = 5/2\allowbreak < 5\allowbreak = 1+2^{3-s}$. For $s=2$, we have $K'(2,t)/K(2,t) < 3 = 1+2^{3-s}$ (see Observation \ref{k2t}).
    
    For the inductive step, fix $s>2$, and suppose that for all $t$,
    \[
        \frac{K'(s-1,t)}{K(s-1,t)} < 1+2^{4-s}.
    \]
    Note that this implies $K'(s-1,t)/K(s-1,t) < 3$.

    If $t=1$, we have $K'(s,1)/K(s,1)=1+3\cdot2^{-s}<1+2^{3-s}$, so we are done. If $t>1$, then by (\ref{K-quotient-formula}), the inductive hypothesis, and Lemma \ref{K-properties}v,
    \begin{align*}
        \frac{K'(s,t)}{K(s,t)} &< 1+3\cdot 2^{-s} + 3\sum_{u=1}^{t-1} \frac{1}{K(s,u)}\\
        &\leq 1+3\cdot 2^{-s} + 3\sum_{u=1}^{t-1} 2^{u-su-1}\\
        &< 1 + 3\cdot 2^{-s} + \frac{3}{2^s-2}\\
        &< 1 + 2^{3-s}.
    \end{align*}
    This completes the inductive step. We may conclude that $1 < K'(s,t)/K(s,t) < 1+2^{3-s}$ for all $s$ and $t$.
\end{proof}

\begin{lemma}
\label{K-Ackermann}
    For all positive integers $s$ and $t$,
    \[
        A(s,t) \leq tK(s,t) < tK'(s,t) < A(s+1,t+2).
    \]
\end{lemma}

\begin{proof}
    The inequality $tK(s,t)<tK'(s,t)$ follows immediately from Lemma \ref{K-properties}iv. Thus, it will suffice to prove that $A(s,t)\leq tK(s,t)$ and $tK'(s,t)<A(s+1,t+2)$.

    First, we will prove that $tK(s,t)\geq A(s,t)$. For $s=1$, this is immediate, as $tK(1,t) = 2t = A(1,t)$.

    For $s>1$, we will prove the stronger statement that $K(s,t) \geq A(s,t)$. We will do so by double induction on $s$ and $t$, with base cases $s=2$ and $t=1$. The base case $s=2$ is immediate, as $K(2,t)=2^{t+1}>2^t=A(2,t)$. The base case $t=1$ is also immediate, as $K(s,1)=2=A(s,1)$.

    For the inductive step, fix $s>2$, and suppose $K(s-1,t)\geq A(s-1,t)$ for all $t$. Fix $t>1$ as well, and suppose $K(s,t-1) \geq A(s,t-1)$. Then
    \begin{align*}
        K(s,t) &= K(s,t-1)K(s-1,K'(s,t-1))\\
        &\geq K(s-1,K'(s,t-1))\\
        &\geq A(s-1,A(s,t-1))\\
        &= A(s,t).
    \end{align*}
    This completes the inductive step. We may conclude that $K(s,t) \geq A(s,t)$ for all $s>1$, so $tK(s,t) \geq A(s,t)$ for all $s$ and $t$.

    Second, we will prove that $tK'(s,t) \leq A(s+1,t+2)-3$. We will again use double induction on $s$ and $t$. The base cases are $s=1,2$ and $t=1$, which we check in order.

    $s=1$:
    \[
        tK'(1,t) = 5t \leq 2^{t+2}-3 = A(2,t+2)-3.
    \]
    $s=2$:
    \[
        tK'(2,t) = t(3\cdot2^{t+1}-5) < 2^{2(t+1)}-3 \leq 2^{A(3,t+1)}-3 = A(3,t+2)-3.
    \]
    $t=1$, using Lemma \ref{ackermann-properties}iv:
    \[
        K'(s,1) = 2^s+3 \leq 2^{s+2}-3 \leq A(s+1,3)-3.
    \]

    For the inductive step, fix $s>2$, and suppose $tK'(s-1,t) \leq A(s,t+2)-3$ for all $t$. Fix $t>1$ as well, and suppose $(t-1)K'(s,t-1) \leq A(s+1,t+1)-3$.

    By the inductive hypothesis,
    \begin{align*}
        t( K'(s,t-1)+1 ) &= tK'(s,t-1)+(t-2)+2\\
        &\leq (2t-2)K'(s,t-1)+2\\
        &\leq 2A(s+1,t+1)-4.
    \end{align*}
    Using the above, Lemma \ref{K-properties}iii, and the inductive hypothesis,
    \begin{align*}
        tK'(s,t) &= t(K'(s,t-1)K(s-1,K'(s,t-1))+K'(s-1,K'(s,t-1)))\\
        &< t( K'(s,t-1)+1 )K'(s-1,K'(s,t-1))\\
        &\leq ( 2A(s+1,t+1)-4 )K'(s-1,A(s+1,t+1)-3)\\
        &< ( A(s+1,t+1)-2 )K'(s-1,A(s+1,t+1)-2)\\
        &\leq A(s,A(s+1,t+1))-3\\
        &= A(s+1,t+2)-3.
    \end{align*}
    This completes the inductive step. We may conclude that $tK'(s,t)\leq A(s+1,t+2)-3$ for all $s$ and $t$.
\end{proof}

\subsection{Construction}
\label{ss-construction}

We will construct a family of CW 3-complexes $X(s,t)$, where $s$ and $t$ range over the positive integers. The complex $X(s,t)$ will be a 3-ball for $t>1$. Subsequently, we will use $X(s,t)$ to construct our desired family of 3-spheres $S(s,t)$.

We will draw $X(s,t)$ from a ``bird's-eye view" in $\mathbb{R}^3$, so the boundary of each facet will have a well-defined top and bottom, as will the boundary of $X(s,t)$ itself for $t>1$. In our drawing, each ridge and facet will have the shape of a convex polygon, giving $X(s,t)$ the appearance of a polytopal complex. This is an optical illusion; we do not necessarily claim that $X(s,t)$ can be realized with convex polytopes in $\mathbb{R}^3$.

\begin{figure}[p]
    \centering
    \tikzmath{\r = 8;}

    \colorlet{root}{lightgray}
    \colorlet{tooth1}{red}
    \colorlet{tooth2}{yellow}
    \colorlet{tooth3}{blue}
    \colorlet{fillet1}{purple}
    \colorlet{fillet2}{orange}
    \colorlet{fillet3}{green}
 
    \begin{tikzpicture}[line cap=round, line join=round, thick]
        \coordinate (o) at (-90:\r);

        \coordinate (a1) at (-50:\r);
        \coordinate (a2) at (-10:\r);
        \coordinate (a3) at (30:\r);
        \coordinate (a4) at (70:\r);
        \coordinate (a5) at (110:\r);
        \coordinate (a6) at (150:\r);
        \coordinate (a7) at (190:\r);
        \coordinate (a8) at (230:\r);

        \coordinate (b1) at (-40:\r);
        \coordinate (c11) at (-30:\r);
        \coordinate (c12) at (-20:\r);

        \coordinate (b2) at (0:\r);
        \coordinate (c21) at (10:\r);
        \coordinate (c22) at (20:\r);

        \coordinate (b3) at (40:\r);

        \coordinate (b5) at (120:\r);
        \coordinate (c51) at (130:\r);
        \coordinate (c52) at (140:\r);

        \coordinate (b6) at (160:\r);
        \coordinate (c61) at (170:\r);
        \coordinate (c62) at (180:\r);

        \coordinate (b7) at (200:\r);

        \draw[{Latex}-{Latex}, shorten <=2cm, shorten >=10pt, line width=2.4] (o) -- (a1) -- node[sloped,below=3pt] {\textbf{profile}} (b1) -- (c11) -- (c12) -- (a2);

        \path[fill=root] (o) -- (a1) -- (a2) -- (a3) -- (a4) -- (a5) -- (a6) -- (a7) -- (a8) -- cycle;

        \path[fill=tooth1] (a1) -- (b1) -- (c12) -- (a2) -- cycle;
        \path[fill=tooth2] (a2) -- (b2) -- (c22) -- (a3) -- cycle;
        \path[fill=tooth3] (a3) -- (b3) -- (a4) -- cycle;
        \path[fill=tooth1] (a5) -- (b5) -- (c52) -- (a6) -- cycle;
        \path[fill=tooth2] (a6) -- (b6) -- (c62) -- (a7) -- cycle;
        \path[fill=tooth3] (a7) -- (b7) -- (a8) -- cycle;

        \path[fill=fillet1] (b1) -- (c11) -- (c12) -- cycle;
        \path[fill=fillet2] (b2) -- (c21) -- (c22) -- cycle;
        \path[fill=fillet1] (b5) -- (c51) -- (c52) -- cycle;
        \path[fill=fillet2] (b6) -- (c61) -- (c62) -- cycle;
        
        \draw[dashed,very thick] (a1) -- (a2) -- (a3);
        \draw[dashed,ultra thick] (a3) -- (a4);
        \draw[dashed,very thick] (a5) -- (a6) -- (a7);
        \draw[dashed,ultra thick] (a7) -- (a8);
        \draw[loosely dashed,very thick] (o) -- (a2);
        \draw[loosely dashed,ultra thick] (o) -- (a3);
        \draw[loosely dashed,ultra thick] (o) -- (a4);
        \draw[loosely dashed,very thick] (o) -- (a5);
        \draw[loosely dashed,very thick] (o) -- (a6);
        \draw[loosely dashed,very thick] (o) -- (a7);
        
        \draw[densely dashed] (b1) -- (a2);
        \draw[densely dashed] (b1) -- (c12);

        \draw[densely dashed] (b2) -- (a3);
        \draw[densely dashed] (b2) -- (c22);

        \draw[densely dashed] (b5) -- (a6);
        \draw[densely dashed] (b5) -- (c52);

        \draw[densely dashed] (b6) -- (a7);
        \draw[densely dashed] (b6) -- (c62);

        \path[fill=tooth2] (c12) -- (a2) -- (b2) -- cycle;
        \path[fill=tooth3] (c22) -- (a3) -- (b3) -- cycle;
        \path[fill=tooth2] (c52) -- (a6) -- (b6) -- cycle;
        \path[fill=tooth3] (c62) -- (a7) -- (b7) -- cycle;

        \path[fill=fillet2] (c11) -- (c12) -- (b2) -- cycle;
        \path[fill=fillet3] (c21) -- (c22) -- (b3) -- cycle;
        \path[fill=fillet2] (c51) -- (c52) -- (b6) -- cycle;
        \path[fill=fillet3] (c61) -- (c62) -- (b7) -- cycle;
        
        \draw (a1) -- (b1) -- (c11) -- (c12) -- (a2) -- (b2) -- (c21) -- (c22) -- (a3) -- (b3) -- (a4);
        \draw (a5) -- (b5) -- (c51) -- (c52) -- (a6) -- (b6) -- (c61) -- (c62) -- (a7);
        \draw[ultra thick] (a7) -- (b7) -- (a8);
        \draw[very thick] (a8) -- (o) -- (a1);
        \draw[very thick] (a4) -- (a5);

        \draw (b2) -- (c11);
        \draw (b2) -- (c12);
        \draw (b3) -- (c21);
        \draw (b3) -- (c22);
        \draw (b6) -- (c51);
        \draw (b6) -- (c52);
        \draw (b7) -- (c61);
        \draw (b7) -- (c62);

        \fill (o) circle (2.5pt) node[below=3pt] {hub};

        \fill (a3) circle (2.5pt);
        \fill (a4) circle (2.5pt);
        
        \fill (c51) circle (2.5pt) (c52) circle (2.5pt) (a6) circle (2.5pt);
        \draw[decorate,decoration={brace,raise=10pt}] (a6) -- node[above=13pt,sloped]{fillet} (c51);

        \path (o) -- (a3) -- (a4) -- cycle;
        \node at (barycentric cs:o=1,a3=1,a4=1) {\textbf{root triangle}};

        \node (roots) at (barycentric cs:o=0.25,a3=1,a4=1) {roots};
        \draw[->,shorten >=1cm] (roots) -- (a3);
        \draw[->,shorten >=1cm] (roots) -- (a4);

        \node (tooth) at (barycentric cs:a7=1,b7=1,a8=1) {};
        \node[right=4pt] at (tooth) {\textbf{tooth}};

        \fill (b7) circle (2.5pt) node[left=3pt] {tip};
    \end{tikzpicture}
    \caption{Part of the boundary of a hypothetical $X(s,t)$ with nine roots, six teeth, and four fillets of three vertices each. Adjacent ridges drawn in different colors (i.e. any pair of adjacent ridges pictured that includes a root triangle and a tooth, or includes neither) necessarily belong to distinct facets.}
    \label{fig-xst-nomenclature}
\end{figure}
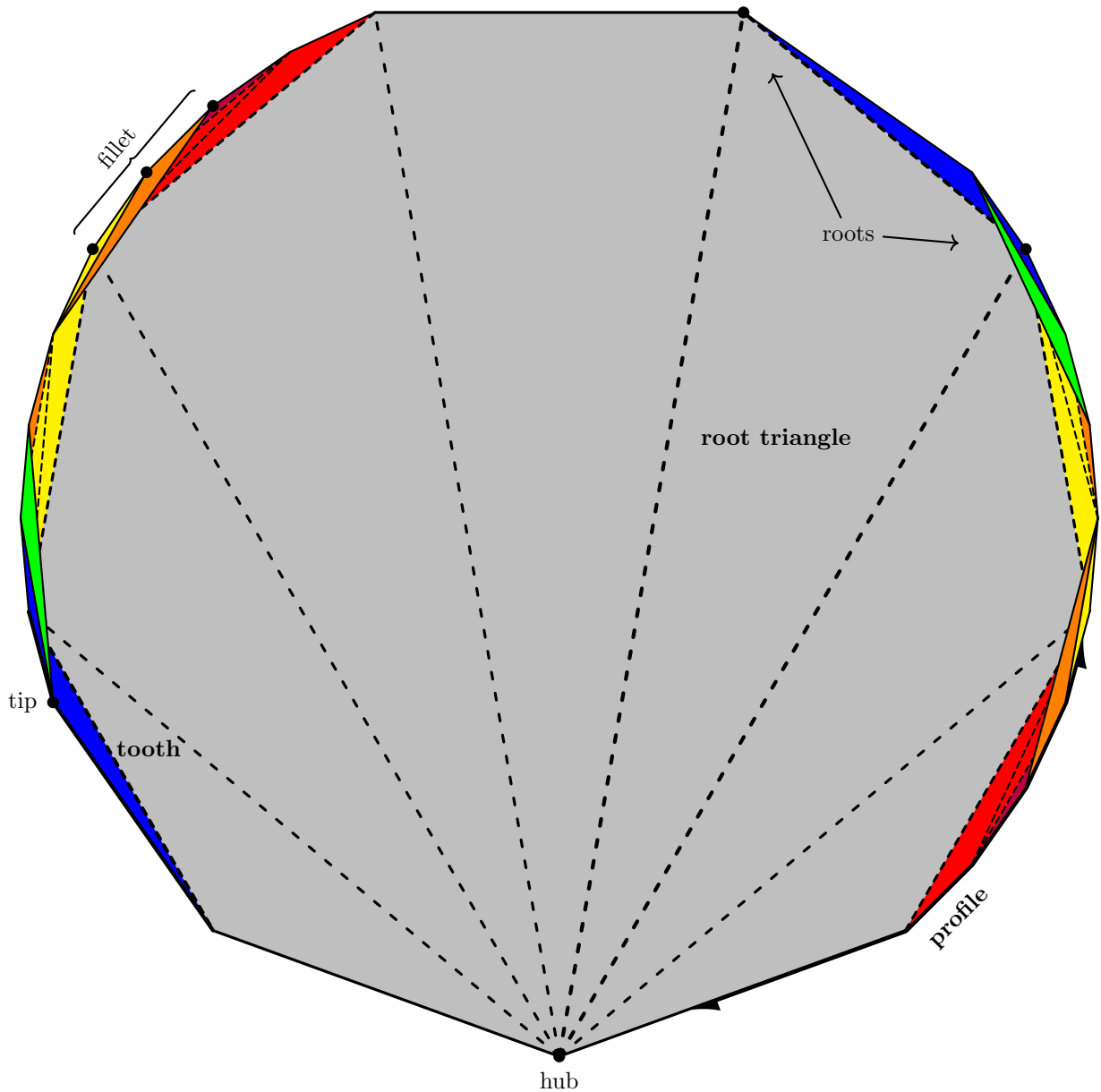

Our construction of $X(s,t)$ will ensure the following (see Figure \ref{fig-xst-nomenclature}).
\begin{enumerate}[label=(\alph*)]
    \item $X(s,t)$ is a 3-ball for $t>1$.\label{property-3ball}
    
    \item Projected onto the page, all vertices of $X(s,t)$ map to points on the unit circle, $X(s,t)$ maps to their convex hull, and the image of each face is the convex hull of the images of its vertices. The projection map acts homeomorphically on the top and bottom of $\partial X(s,t)$, as well as the top and bottom of the boundary of each facet. We call the boundary of the image of $X(s,t)$ the \emph{profile} of $X(s,t)$.\label{property-circle}
    
    \item The bottom of $\partial X(s,t)$ contains a fan of triangles $\bigtriangleup a_1 a_2 a_3, \allowbreak \bigtriangleup a_1 a_3 a_4, \allowbreak \ldots, \allowbreak \bigtriangleup a_1 a_{K'(s,t)-1} a_{K'(s,t)}$, where vertices $a_1, \ldots, a_{K'(s,t)}$ appear counterclockwise along the profile of $X(s,t)$, and $a_{K'(s,t)},a_1,a_2$ are consecutive. We call vertex $a_1$ the \emph{hub}, vertices $a_1,\ldots,a_{K'(s,t)}$ the \emph{roots}, and these triangles the \emph{root triangles} of $X(s,t)$.\label{property-root}
   
    \item For each index $1 < i < K'(s,t)$, one of the following holds:
    \begin{itemize}
        \item $a_i, a_{i+1}$ are adjacent on the profile of $X(s,t)$,
        \item the bottom of $\partial X(s,t)$ contains a triangle $\bigtriangleup a_i b_i a_{i+1}$ such that $a_i,b_i,a_{i+1}$ appear consecutively counterclockwise on the profile of $X(s,t)$, or
        \item the bottom of $\partial X(s,t)$ contains triangles $\bigtriangleup a_i b_i a_{i+1}$ and $\bigtriangleup b_i d_i a_{i+1}$ belonging to a common facet, vertices $a_i,b_i,d_i,a_{i+1}$ appear counterclockwise on the profile of $X(s,t)$, vertices $a_i,b_i$ are consecutive, vertices $d_i,a_{i+1}$ are consecutive, and vertices $a_i,d_i$ do not belong to a common ridge.
    \end{itemize}
    In the latter two cases, we call $\bigtriangleup a_i b_i a_{i+1}$ a \emph{tooth} and $b_i$ its \emph{tip}.\label{property-tooth}
    
    \item For exactly $K(s,t)$ many indices $2 < i < K'(s,t)$, there is a list of consecutive vertices $b_{i-1}, c_1,\allowbreak \ldots,\allowbreak c_{t-1},\allowbreak c_t=a_i,\allowbreak b_i$ ordered counterclockwise along the profile of $X(s,t)$ such that
    \begin{itemize}
        \item $c_t$ is a root, and $b_{i-1},b_i$ are tips of consecutive teeth that both contain $c_t$,
        \item triangles $\bigtriangleup b_{i-1} c_1 c_2, \bigtriangleup b_{i-1} c_2 c_3, \ldots, \bigtriangleup b_{i-1} c_{t-1} c_t$ are on the bottom of $\partial X(s,t)$,
        \item triangles $\bigtriangleup b_i c_1 c_2, \bigtriangleup b_i c_2 c_3, \ldots, \bigtriangleup b_i c_{t-1} c_t$ are on the top of $\partial X(s,t)$,
        \item each edge $b_ic_1, \ldots, b_ic_t$ belongs to exactly one ridge outside $\{\bigtriangleup b_i c_1 c_2, \bigtriangleup b_i c_2 c_3, \ldots, \bigtriangleup b_i c_{t-1} c_t\}$, which is a triangle whose additional vertex appears counterclockwise from $b_i$, and
        \item if we pick two nonconsecutive vertices from the list $(c_1, \ldots, c_t)$, these two vertices will never belong to a common facet.
    \end{itemize}
    We call the list of vertices $(c_1, \ldots, c_t)$ a \emph{fillet} of $X(s,t)$.\label{property-fillet}

    \item The hub $a_1$ does not belong to a common facet with any vertex outside $\{a_1,\ldots,a_{K'(s,t)}\}$.\label{property-hub}
\end{enumerate}

\begin{observation}
    $X(s,t)$ has $t$ many vertices in each fillet, $K(s,t)$ many fillets, and $K'(s,t)$ many roots including the hub.
\end{observation}

\begin{construction}
\label{ball-construction}
For all positive integers $s$ and $t$, we define $X(s,t)$ recursively as follows.
\begin{itemize}
    \item $X(s,1)$ is a triangulated disc consisting of $2^s+1$ many root triangles and $2^s+1$ many teeth (Figure \ref{fig-xs1}). We regard this as a ``degenerate" 3-complex with no facets, so $\partial X(s,1)=X(s,1)$, and the ``top" and ``bottom" of $X(s,1)$ both equal $X(s,1)$ in its entirety. We let each root contained in two teeth be a fillet. (The second, third, and fifth statements in \ref{property-fillet} are here vacuously true.)
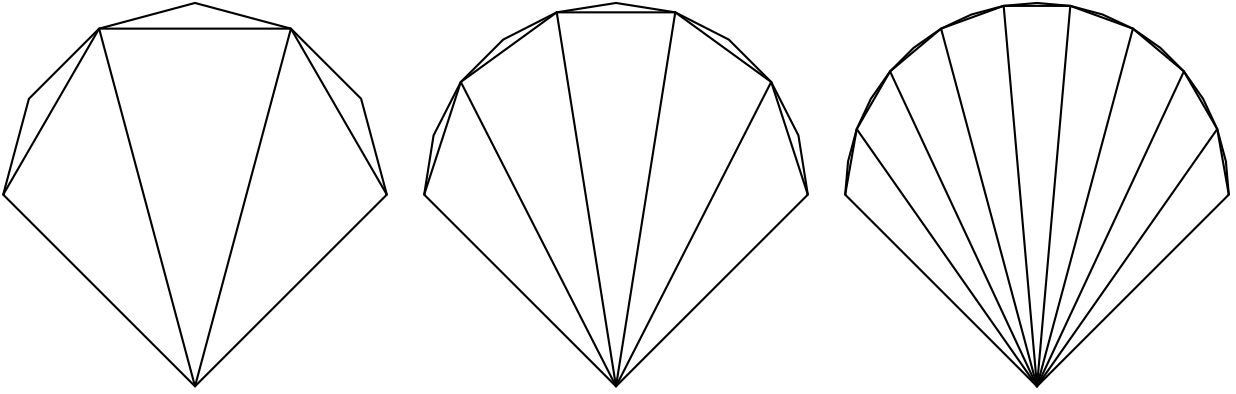
\begin{figure}[t]
    \centering
    \tikzmath{\r = 2.55;}
    \begin{tikzpicture}[line cap=round, line join=round, thick]
        \coordinate (o) at (-90:\r);
        \coordinate (a1) at (0:\r);
        \coordinate (a2) at (60:\r);
        \coordinate (a3) at (120:\r);
        \coordinate (a4) at (180:\r);
        \coordinate (b1) at (30:\r);
        \coordinate (b2) at (90:\r);
        \coordinate (b3) at (150:\r);

        \draw (o) -- (a1) -- (a2) -- (a3) -- (a4) -- cycle;
        \draw (a1) -- (b1) -- (a2) -- (b2) -- (a3) -- (b3) -- (a4);
        \draw (o) -- (a2);
        \draw (o) -- (a3);
    \end{tikzpicture}
    \quad
    \begin{tikzpicture}[line cap=round, line join=round, thick]
        \coordinate (o) at (-90:\r);
        \coordinate (a1) at (0:\r);
        \coordinate (a2) at (36:\r);
        \coordinate (a3) at (72:\r);
        \coordinate (a4) at (108:\r);
        \coordinate (a5) at (144:\r);
        \coordinate (a6) at (180:\r);
        \coordinate (b1) at (18:\r);
        \coordinate (b2) at (54:\r);
        \coordinate (b3) at (90:\r);
        \coordinate (b4) at (126:\r);
        \coordinate (b5) at (162:\r);

        \draw (o) -- (a1) -- (a2) -- (a3) -- (a4) -- (a5) -- (a6) -- cycle;
        \draw (a1) -- (b1) -- (a2) -- (b2) -- (a3) -- (b3) -- (a4) -- (b4) -- (a5) -- (b5) -- (a6);
        \draw (o) -- (a2);
        \draw (o) -- (a3);
        \draw (o) -- (a4);
        \draw (o) -- (a5);
    \end{tikzpicture}
    \quad
    \begin{tikzpicture}[line cap=round, line join=round, thick]
        \coordinate (o) at (-90:\r);
        \coordinate (a1) at (0:\r);
        \coordinate (a2) at (20:\r);
        \coordinate (a3) at (40:\r);
        \coordinate (a4) at (60:\r);
        \coordinate (a5) at (80:\r);
        \coordinate (a6) at (100:\r);
        \coordinate (a7) at (120:\r);
        \coordinate (a8) at (140:\r);
        \coordinate (a9) at (160:\r);
        \coordinate (a10) at (180:\r);
        \coordinate (b1) at (10:\r);
        \coordinate (b2) at (30:\r);
        \coordinate (b3) at (50:\r);
        \coordinate (b4) at (70:\r);
        \coordinate (b5) at (90:\r);
        \coordinate (b6) at (110:\r);
        \coordinate (b7) at (130:\r);
        \coordinate (b8) at (150:\r);
        \coordinate (b9) at (170:\r);

        \draw (o) -- (a1) -- (a2) -- (a3) -- (a4) -- (a5) -- (a6) -- (a7) -- (a8) -- (a9) -- (a10) -- cycle;
        \draw (a1) -- (b1) -- (a2) -- (b2) -- (a3) -- (b3) -- (a4) -- (b4) -- (a5) -- (b5) -- (a6) -- (b6) -- (a7) -- (b7) -- (a8) -- (b8) -- (a9) -- (b9) -- (a10);
        \draw (o) -- (a2);
        \draw (o) -- (a3);
        \draw (o) -- (a4);
        \draw (o) -- (a5);
        \draw (o) -- (a6);
        \draw (o) -- (a7);
        \draw (o) -- (a8);
        \draw (o) -- (a9);
    \end{tikzpicture}
    \caption{From left to right, CW complexes $X(s,1)$ for $s=1,2,3$.}
    \label{fig-xs1}
\end{figure}

    \item For $t>1$, $X(1,t)$ is a 3-ball with hub $a_1$, roots $a_1,\ldots,a_5$ in counterclockwise order, teeth $\bigtriangleup a_2b_2a_3,\allowbreak\bigtriangleup a_3b_3a_4,\allowbreak\bigtriangleup a_4b_4a_5$, and fillets $(c_1,\ldots,c_{t-1},a_3),(c'_1,\ldots,c'_{t-1},a_4)$ (Figure \ref{fig-x1t}). It has the following facets:
    \begin{itemize}
        \item a facet on vertices $a_1a_2a_3a_4a_5$ (gray), isomorphic to a connected sum of simplices $a_1a_2a_3a_5$ and $a_1a_3a_4a_5$,
        \item a facet on vertices $a_2b_2c_{t-1}a_3a_5$ (blue), isomorphic to a connected sum of simplices $a_2b_2a_3a_5$ and $b_2c_{t-1}a_3a_5$,
        \item a facet on vertices $c_{t-1}a_3b_3c'_{t-1}a_4a_5$ (pink), isomorphic to a connected sum of simplices $c_{t-1}a_3b_3a_5$, $a_3b_3a_4a_5$, and $b_3c'_{t-1}a_4a_5$,
        \item a facet isomorphic to a simplex on vertices $c'_{t-1}a_4b_4a_5$ (yellow), and
        \item for each $i=1,\ldots,t-2$, four facets isomorphic to simplices on vertices $b_2c_ic_{i+1}a_5$, $c_ic_{i+1}b_3a_5$, $b_3c'_ic'_{i+1}a_5$, and $c'_ic'_{i+1}b_4a_5$ (not colored).
    \end{itemize}
\begin{figure}[t]
    \centering
    \tikzmath{\r = 2.4;\opac = 0.75;}
    
    \colorlet{root}{gray}
    \colorlet{tooth1}{cyan}
    \colorlet{tooth2}{magenta}
    \colorlet{tooth3}{yellow}
    
    \begin{tikzpicture}[line cap=round, line join=round, thick]
        \coordinate (o) at (-90:\r);
        \coordinate (x1) at (0:\r);
        \coordinate (x2) at (15:\r);
        \coordinate (x3) at (45:\r);
        \coordinate (x4) at (60:\r);
        \coordinate (x5) at (75:\r);
        \coordinate (x6) at (105:\r);
        \coordinate (x7) at (120:\r);
        \coordinate (x8) at (135:\r);
        \coordinate (x9) at (180:\r);

        \draw[loosely dashed,very thick] (o) -- (x4);
        \draw[loosely dashed,very thick] (o) -- (x7);
        \path[fill=root, opacity=\opac] (o) -- (x1) -- (x4) -- (x7) -- (x9) -- cycle;
        \draw[very thick] (x9) -- (o) -- (x1);
        \draw[dashed,very thick] (x1) -- (x4) -- (x7);

        \draw[densely dashed] (x2) -- (x4);
        \path[fill=tooth1, opacity=\opac] (x1) -- (x2) -- (x3) -- (x4) -- (x9) -- cycle;
        \draw[very thick] (x1) -- (x9);
        \draw (x1) -- (x2) -- (x9);
        \draw[dashed,very thick] (x4) -- (x9);

        \draw[densely dashed] (x5) -- (x7);
        \path[fill=tooth2, opacity=\opac] (x3) -- (x4) -- (x5) -- (x6) -- (x7) -- (x9) -- cycle;
        \draw (x3) -- (x4) -- (x5) -- cycle;
        \draw (x5) -- (x9);
        \draw[dashed,very thick] (x7) -- (x9);

        \path[fill=tooth3, opacity=\opac] (x6) -- (x7) -- (x8) -- (x9) -- cycle;
        \draw (x6) -- (x7) -- (x8) -- (x9);
        \draw (x6) -- (x8);

        \draw (x2) -- (x3) -- (x9);
        \draw (x5) -- (x6) -- (x9);

        \node[below] at (o) {$a_1$};
        \node[right=-2pt] at (x1) {$\mathrlap{a_2}$};
        \node[right=-1pt] at (x2) {$\mathrlap{b_2}$};
        \node[right] at (x3) {$c_1$};
        \node[above=4pt,right=-2pt] at (x4) {$a_3$};
        \node[above] at (x5) {$b_3$};
        \node[above] at (x6) {$c'_1$};
        \node[above=4pt,left=-2pt] at (x7) {$a_4$};
        \node[left] at (x8) {$b_4$};
        \node[left=-2pt] at (x9) {$a_5$};
    \end{tikzpicture}
    \quad
    \begin{tikzpicture}[line cap=round, line join=round, thick]
        \coordinate (o) at (-90:\r);
        \coordinate (x1) at (0:\r);
        \coordinate (x2) at (15:\r);
        \coordinate (x3) at (45:\r);
        \coordinate (x4) at (60:\r);
        \coordinate (x5) at (75:\r);
        \coordinate (x6) at (105:\r);
        \coordinate (x7) at (120:\r);
        \coordinate (x8) at (135:\r);
        \coordinate (x9) at (180:\r);

        \draw[loosely dashed,very thick] (o) -- (x4);
        \draw[loosely dashed,very thick] (o) -- (x7);
        \path[fill=root, opacity=\opac] (o) -- (x1) -- (x4) -- (x7) -- (x9) -- cycle;
        \draw[very thick] (x9) -- (o) -- (x1);
        \draw[dashed,very thick] (x1) -- (x4) -- (x7);

        \draw[densely dashed] (x2) -- (x4);
        \path[fill=tooth1, opacity=\opac] (x1) -- (x2) -- (x3) -- (x4) -- (x9) -- cycle;
        \draw[very thick] (x1) -- (x9);
        \draw (x1) -- (x2) -- (x9);
        \draw[dashed,very thick] (x4) -- (x9);

        \draw[densely dashed] (x5) -- (x7);
        \path[fill=tooth2, opacity=\opac] (x3) -- (x4) -- (x5) -- (x6) -- (x7) -- (x9) -- cycle;
        \draw (x3) -- (x4) -- (x5) -- cycle;
        \draw (x5) -- (x9);
        \draw[dashed,very thick] (x7) -- (x9);

        \path[fill=tooth3, opacity=\opac] (x6) -- (x7) -- (x8) -- (x9) -- cycle;
        \draw (x6) -- (x7) -- (x8) -- (x9);
        \draw (x6) -- (x8);

        \draw[densely dashed] (x2) -- (x3) -- (x9);
        \draw[densely dashed] (x5) -- (x6) -- (x9);

        \coordinate (c1) at (35:\r);
        \coordinate (d1) at (95:\r);

        \draw (x2) -- (c1) -- (x9);
        \draw (x3) -- (c1) -- (x5);

        \draw (x5) -- (d1) -- (x9);
        \draw (x6) -- (d1) -- (x8);

        \node[below] at (o) {$\phantom{a_1}$};
        \node[right] at (x3) {$c_2$};
        \node[right] at (c1) {$c_1$};
        \node[above] at (x6) {$c'_2$};
        \node[above] at (d1) {$c'_1$};
    \end{tikzpicture}
    \quad
    \begin{tikzpicture}[line cap=round, line join=round, thick]
        \coordinate (o) at (-90:\r);
        \coordinate (x1) at (0:\r);
        \coordinate (x2) at (15:\r);
        \coordinate (x3) at (45:\r);
        \coordinate (x4) at (60:\r);
        \coordinate (x5) at (75:\r);
        \coordinate (x6) at (105:\r);
        \coordinate (x7) at (120:\r);
        \coordinate (x8) at (135:\r);
        \coordinate (x9) at (180:\r);

        \draw[loosely dashed,very thick] (o) -- (x4);
        \draw[loosely dashed,very thick] (o) -- (x7);
        \path[fill=root, opacity=\opac] (o) -- (x1) -- (x4) -- (x7) -- (x9) -- cycle;
        \draw[very thick] (x9) -- (o) -- (x1);
        \draw[dashed,very thick] (x1) -- (x4) -- (x7);

        \draw[densely dashed] (x2) -- (x4);
        \path[fill=tooth1, opacity=\opac] (x1) -- (x2) -- (x3) -- (x4) -- (x9) -- cycle;
        \draw[very thick] (x1) -- (x9);
        \draw (x1) -- (x2) -- (x9);
        \draw[dashed,very thick] (x4) -- (x9);

        \draw[densely dashed] (x5) -- (x7);
        \path[fill=tooth2, opacity=\opac] (x3) -- (x4) -- (x5) -- (x6) -- (x7) -- (x9) -- cycle;
        \draw (x3) -- (x4) -- (x5) -- cycle;
        \draw (x5) -- (x9);
        \draw[dashed,very thick] (x7) -- (x9);

        \path[fill=tooth3, opacity=\opac] (x6) -- (x7) -- (x8) -- (x9) -- cycle;
        \draw (x6) -- (x7) -- (x8) -- (x9);
        \draw (x6) -- (x8);

        \draw[densely dashed] (x2) -- (x3) -- (x9);
        \draw[densely dashed] (x5) -- (x6) -- (x9);

        \coordinate (c1) at (35:\r);
        \coordinate (c0) at (25:\r);
        \coordinate (d1) at (95:\r);
        \coordinate (d0) at (85:\r);

        \draw[densely dashed] (x2) -- (c1) -- (x9);
        \draw (x3) -- (c1) -- (x5);

        \draw (x2) -- (c0) -- (x9);
        \draw (c1) -- (c0) -- (x5);

        \draw[densely dashed] (x5) -- (d1) -- (x9);
        \draw (x6) -- (d1) -- (x8);

        \draw (x5) -- (d0) -- (x9);
        \draw (d1) -- (d0) -- (x8);

        \node[below] at (o) {$\phantom{a_1}$};
        \node[right] at (x3) {$c_3$};
        \node[right] at (c1) {$c_2$};
        \node[right] at (c0) {$c_1$};
        \node[above] at (x6) {$c'_3$};
        \node[above] at (d1) {$c'_2$};
        \node[above] at (d0) {$c'_1$};
    \end{tikzpicture}
    \caption{From left to right, CW complexes $X(1,t)$ for $t=2,3,4$.}
    \label{fig-x1t}
\end{figure}
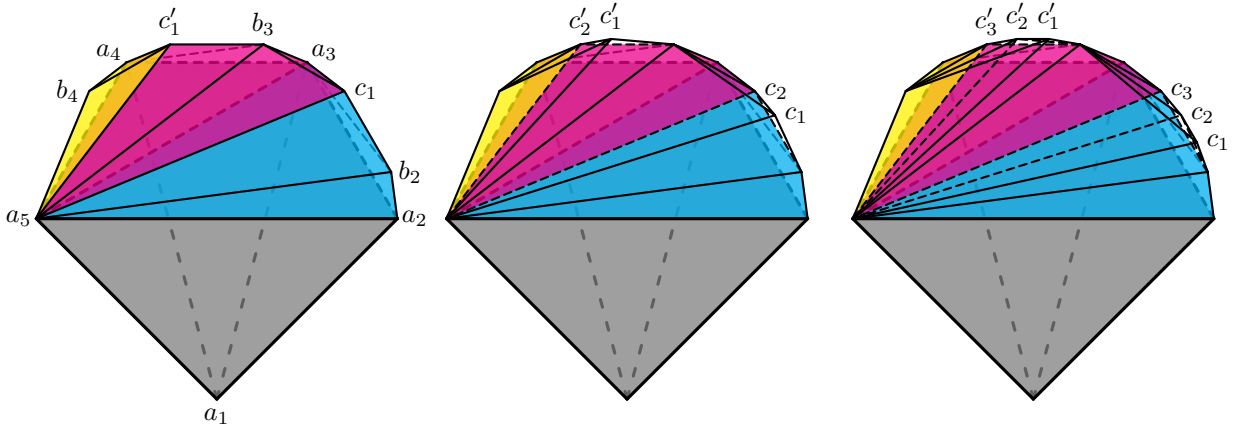

    \item For $s,t>1$, $X(s,t)$ is constructed from $X(s-1,K'(s,t-1))$ and $K(s-1,K'(s,t-1))$ many copies of $X(s,t-1)$. Let $o$ be the hub of $X(s-1,K'(s,t-1))$. First, we reflect $X(s-1,K'(s,t-1))$ horizontally. Next, for each fillet of $X(s-1,K'(s,t-1))$, we do the following (see Figure \ref{fig-xst-recursion}).
    \begin{enumerate}
        \item Label the vertices in the fillet $c_1,\ldots,c_{K'(s,t-1)}$ going \emph{clockwise} (as we have reflected $X(s-1,K'(s,t-1))$, this is consistent with our notation in \ref{property-fillet}). Let $b$ be the tip clockwise of our fillet, $b'$ the tip counterclockwise of our fillet, and $a$ the root counterclockwise of $b'$, so $\bigtriangleup c_{K'(s,t-1)} b' a$ is a tooth and $\bigtriangleup o c_{K'(s,t-1)} a$ a root triangle. By \ref{property-fillet}, triangles $\bigtriangleup b c_2 c_1, \allowbreak \bigtriangleup b c_3 c_2, \allowbreak \ldots, \allowbreak \bigtriangleup b c_{K'(s,t-1)} c_{K'(s,t-1)-1}$ are on the top of $\partial X(s-1,K'(s,t-1))$, while triangles $\bigtriangleup b' c_2 c_1, \allowbreak \bigtriangleup b' c_3 c_2, \allowbreak \ldots, \allowbreak \bigtriangleup b' c_{K'(s,t-1)} c_{K'(s,t-1)-1}$ are on the bottom of $\partial X(s-1,K'(s,t-1))$.

        For $i=1, \ldots, K'(s,t-1)-1$, let $Y_i$ be the facet containing $\bigtriangleup b c_{i+1} c_i$, and let $Z_i$ be the unique ridge such that $bc_i \subset Z_i \subset Y_i$ and $Z_i \neq \bigtriangleup b c_{i+1} c_i$. By \ref{property-fillet}, each $Z_i$ is a triangle lying to the left of $\overrightarrow{bc_i}$, and $Y_{i-1} \cap Y_i=Z_i$ for $i>1$.\label{step-1}

        \item On the bottom, attach a new facet isomorphic to a simplex on vertices $o,a,b',c_{K'(s,t-1)}$, introducing ridges $\bigtriangleup o b' a$ and $\bigtriangleup o c_{K'(s,t-1)} b'$. By \ref{property-hub}, vertices $o$ and $b'$ were not previously on a common facet, so we maintain strong regularity.\label{step-2}

        \item On the bottom, attach a new facet on vertices $o,b',c_1,\ldots, c_{K'(s,t-1)}$, isomorphic to a connected sum of simplices $oc_2c_1b',\allowbreak oc_3c_2b',\allowbreak\ldots,\allowbreak oc_{K'(s,t-1)}c_{K'(s,t-1)-1}b'$, introducing ridges $\bigtriangleup oc_1b',\allowbreak\bigtriangleup oc_2c_1,\allowbreak\bigtriangleup oc_3c_2,\allowbreak\ldots,\allowbreak\bigtriangleup oc_{K'(s,t-1)}c_{K'(s,t-1)-1}$. By \ref{property-fillet}-\ref{property-hub}, no two nonconsecutive vertices in the list $(c_1,\allowbreak\ldots,\allowbreak c_{K'(s,t-1)},\allowbreak o)$ were originally on a common facet, so we maintain strong regularity. Vertices $b',\allowbreak c_1,\allowbreak \ldots,\allowbreak c_{K'(s,t-1)-1}$ will become roots of $X(s,t)$. Ridges $\bigtriangleup o b' a, \allowbreak \bigtriangleup o c_1 b', \bigtriangleup o c_2 c_1, \allowbreak \bigtriangleup o c_3 c_2, \allowbreak \ldots, \allowbreak \bigtriangleup o c_{K'(s,t-1)} c_{K'(s,t-1)-1}$ will replace $\bigtriangleup o c_{K'(s,t-1)} a$ as root triangles.\label{step-3}

        \item Draw new vertices $\alpha_1, \ldots, \alpha_{K'(s,t-1)-1}$ and $\beta_1, \ldots, \beta_{K'(s,t-1)-2}$ so that projected onto the page, the following vertices lie in clockwise order on the unit circle:
        \[
            c_1,\ \alpha_1,\ \beta_1,\ \ldots,\ c_{K'(s,t-1)-2},\ \alpha_{K'(s,t-1)-2},\ \beta_{K'(s,t-1)-2},\ c_{K'(s,t-1)-1},\ \alpha_{K'(s,t-1)-1},\ c_{K'(s,t-1)}.
        \]
        For each $i=1,\ldots,K'(s,t-1)-1$, extend ridge $Z_i$ to include $\alpha_i$, replacing edge $bc_i$ with edges $b\alpha_i,\alpha_ic_i$. For $i < K'(s,t-1)-1$, extend facet $Y_i$ to include $\alpha_i$, $\beta_i$, and $\alpha_{i+1}$, keeping the extended $Z_i$ and $Z_{i+1}$ as faces, and replacing $\bigtriangleup b c_{i+1} c_i$ with triangles $\bigtriangleup b \alpha_{i+1} \alpha_i, \allowbreak \bigtriangleup \alpha_{i+1} \beta_i \alpha_i, \allowbreak \bigtriangleup \alpha_{i+1} c_{i+1} \beta_i$ on the top and $\bigtriangleup c_{i+1} \beta_i c_i, \allowbreak \bigtriangleup \beta_i \alpha_i c_i$ on the bottom of $\partial Y_i$. For $i=K'(s,t-1)-1$, extend facet $Y_i$ to include $\alpha_i$, keeping the extended $Z_i$ as a face, and replacing $\bigtriangleup b c_{i+1} c_i$ with triangles $\bigtriangleup b c_{i+1} \alpha_i$ on the top and $\bigtriangleup c_{i+1} \alpha_i c_i$ on the bottom of $\partial Y_i$. Triangles
        \[
            \bigtriangleup c_2 \beta_1 c_1,\ \ldots,\ \bigtriangleup c_{K'(s,t-1)-1} \beta_{K'(s,t-1)-2} c_{K'(s,t-1)-2},\ \bigtriangleup c_{K'(s,t-1)} \alpha_{K'(s,t-1)-1} c_{K'(s,t-1)-1}
        \]
        will replace $\triangle c_{K'(s,t-1)} b' a$ as teeth of $X(s,t)$.\label{step-4}

        \item Attach a copy of $X(s,t-1)$ to our complex by identifying the hub of $X(s,t-1)$ with vertex $b$, the non-hub roots of $X(s,t-1)$ with vertices $\alpha_1, \ldots, \alpha_{K'(s,t-1)-1}$, and the teeth of $X(s,t-1)$ with a subset of the triangles
        \[
            \bigtriangleup \alpha_2 \beta_1 \alpha_1,\ \ldots,\ \bigtriangleup \alpha_{K'(s,t-1)-1} \beta_{K'(s,t-1)-2} \alpha_{K'(s,t-1)-2}.
        \]
        By construction, no two nonconsecutive vertices in the list $(\alpha_1, \ldots, \alpha_{K'(s,t-1)-1})$ were previously on a common facet, so our hub and root identifications preserve strong regularity. By \ref{property-hub}, no tip of $X(s,t-1)$ was previously on a common facet with the hub of $X(s,t-1)$, so our tooth identifications preserve strong regularity as well.

        In each fillet of $X(s,t-1)$, we have identified the last vertex with $\alpha_i$ for some $1<i<K'(s,t-1)-1$. Append to each fillet of $X(s,t-1)$ the corresponding vertex $c_i$ to obtain a fillet of $X(s,t)$.\label{step-5}
    \end{enumerate}
    We repeat these steps $K(s-1,K'(s,t-1))$ many times, once for each fillet of $X(s-1,K'(s,t-1))$. The original complex $X(s-1,K'(s,t-1))$ contributes $K'(s-1,K'(s,t-1))$ many roots, and each repetition adds $K(s,t-1)$ many fillets and $K'(s,t-1)$ many roots. The resulting complex therefore has $K(s,t-1)K(s-1,K'(s,t-1))=K(s,t)$ many fillets and $K'(s,t-1)K(s-1,K'(s,t-1))+K'(s-1,K'(s,t-1))=K'(s,t)$ many roots, as desired. Each fillet contains $t-1$ many vertices from $X(s,t-1)$ and one new vertex, for the desired total of $t$.

    Finally, consider any tooth $\bigtriangleup a b a'$ of $X(s-1,K'(s,t-1))$ remaining on the boundary of our complex, where $a,b,a'$ appear counterclockwise along the profile. If $a,b,a'$ are consecutive, we simply let $\bigtriangleup aba'$ be a tooth of $X(s,t)$. Otherwise, by \ref{property-tooth}, the bottom of our complex contains another triangle $\bigtriangleup a d b$ belonging to the same facet, vertices $a,d$ are consecutive, vertices $b,a'$ are consecutive, and vertices $a',d$ do not belong to a common ridge. In this case, we delete edge $ab$ and the two triangles, draw a new edge $a'd$ and triangles $\bigtriangleup a d a', \bigtriangleup dba'$, and let $\bigtriangleup a d a'$ be a tooth of $X(s,t)$. We repeat this for every tooth of $X(s-1,K'(s,t-1))$ remaining on the boundary of our complex.

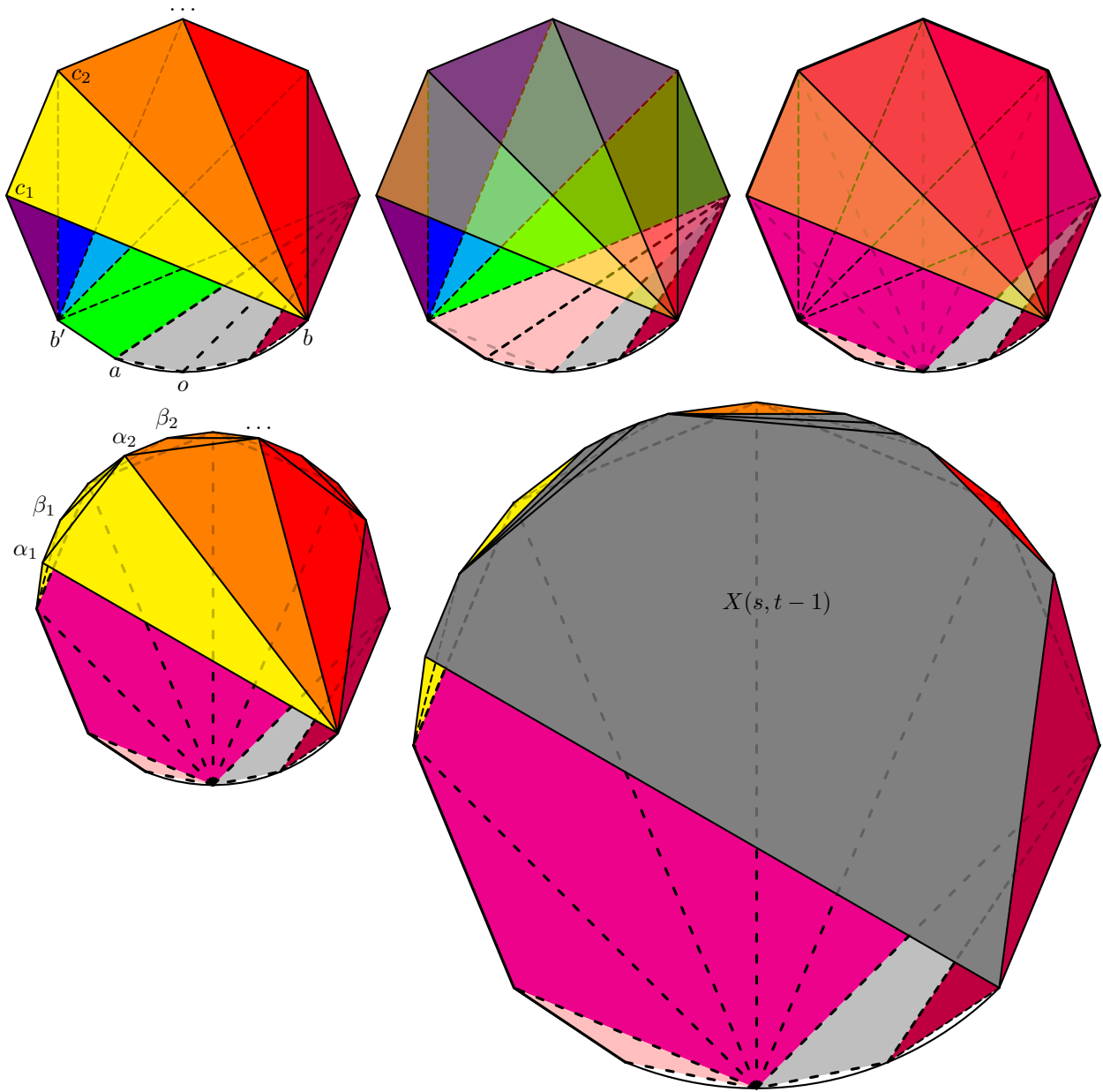
\begin{figure}[p]
    \centering
    \tikzmath{\r = 2.66;\transparent = 0.5;\translucent = 0.75;}
    \colorlet{root}{lightgray}
    \colorlet{back1}{green}
    \colorlet{back2}{cyan}
    \colorlet{back3}{blue}
    \colorlet{back4}{violet}
    \colorlet{front1}{purple}
    \colorlet{front2}{red}
    \colorlet{front3}{orange}
    \colorlet{front4}{yellow}
    \colorlet{facet1}{pink}
    \colorlet{facet2}{magenta}
    \colorlet{complex}{gray}
    \begin{tikzpicture}[line cap=round,line join=round, thick,baseline=(current bounding box.north)]
        \coordinate (o) at (-90:\r);
        \coordinate (a) at (-67.5:\r);
        \coordinate (b) at (-45:\r);
        \coordinate (c5) at (0:\r);
        \coordinate (c4) at (45:\r);
        \coordinate (c3) at (90:\r);
        \coordinate (c2) at (135:\r);
        \coordinate (c1) at (180:\r);
        \coordinate (b') at (225:\r);
        \coordinate (a') at (247.5:\r);

        \path[fill=root] (o) -- (a) -- (c5) -- (a') -- cycle;
        \path[fill=back1] (a') -- (c5) -- (c4) -- (b') -- cycle;
        \path[fill=back2] (b') -- (c4) -- (c3) -- cycle;
        \path[fill=back3] (b') -- (c3) -- (c2) -- cycle;
        \path[fill=back4] (b') -- (c2) -- (c1) -- cycle;
        \path[fill=front1] (a) -- (b) -- (c5) -- cycle;

        \draw[loosely dashed,very thick] (a') -- (o) -- (a);
        \draw[loosely dashed,very thick] (o) -- (c5);
        \draw[dashed, very thick] (a) -- (c5) -- (a');
        \draw[densely dashed] (a) -- (b);
        \draw[densely dashed] (b') -- (c5);
        \draw[densely dashed] (b') -- (c4);
        \draw[densely dashed] (b') -- (c3);
        \draw[densely dashed] (b') -- (c2);
        \draw (a') arc (-112.5:-45:\r);

        \path[fill=front1] (b) -- (c5) -- (c4) -- cycle;
        \path[fill=front2] (b) -- (c4) -- (c3) -- cycle;
        \path[fill=front3] (b) -- (c3) -- (c2) -- cycle;
        \path[fill=front4] (b) -- (c2) -- (c1) -- cycle;

        \draw (b) -- (c5) -- (c4) -- (c3) -- (c2) -- (c1) -- (b') -- (a');
        \draw (b) -- (c4);
        \draw (b) -- (c3);
        \draw (b) -- (c2);
        \draw (b) -- (c1);

        \begin{scope}
            \clip (b) -- (c5) -- (c4) -- (c3) -- (c2) -- (c1) -- cycle;

            \draw[densely dashed,opacity=0.25] (b') -- (c5);
            \draw[densely dashed,opacity=0.25] (b') -- (c4);
            \draw[densely dashed,opacity=0.25] (b') -- (c3);
            \draw[densely dashed,opacity=0.25] (b') -- (c2);
            \draw[loosely dashed,very thick,opacity=0.25] (o) -- (c5);
            \draw[dashed, very thick,opacity=0.25] (a) -- (c5) -- (a');
        \end{scope}

        \node[below] at (o) {$o$};
        \node[below] at (a') {$a$};
        \node[below] at (b') {$b'$};
        \node[above=2pt,right] at (c1) {$c_1$};
        \node[below=2.5pt,right=2pt] at (c2) {$c_2$};
        \node[above] at (c3) {$\ldots$};

        \node[below] at (b) {$b$};
    \end{tikzpicture}
    \
    \begin{tikzpicture}[line cap=round,line join=round, thick,baseline=(current bounding box.north)]
        \coordinate (o) at (-90:\r);
        \coordinate (a) at (-67.5:\r);
        \coordinate (b) at (-45:\r);
        \coordinate (c5) at (0:\r);
        \coordinate (c4) at (45:\r);
        \coordinate (c3) at (90:\r);
        \coordinate (c2) at (135:\r);
        \coordinate (c1) at (180:\r);
        \coordinate (b') at (225:\r);
        \coordinate (a') at (247.5:\r);

        \path[fill=root] (o) -- (a) -- (c5) -- cycle;
        \path[fill=back1] (c5) -- (c4) -- (b') -- cycle;
        \path[fill=back2] (b') -- (c4) -- (c3) -- cycle;
        \path[fill=back3] (b') -- (c3) -- (c2) -- cycle;
        \path[fill=back4] (b') -- (c2) -- (c1) -- cycle;
        \path[fill=front1] (a) -- (b) -- (c5) -- cycle;

        \path[fill=facet1] (o) -- (c5) -- (b') -- (a') -- cycle;
        \draw[loosely dashed,very thick,opacity=0.25] (o) -- (b');

        \draw[loosely dashed,very thick] (a') -- (o) -- (a);
        \draw[loosely dashed,very thick] (o) -- (c5);
        \draw[dashed,very thick] (a) -- (c5) -- (a');
        \draw[densely dashed] (a) -- (b);
        \draw[densely dashed] (b') -- (c5);
        \draw[densely dashed] (b') -- (c4);
        \draw[densely dashed] (b') -- (c3);
        \draw[densely dashed] (b') -- (c2);
        \draw (a') arc (-112.5:-45:\r);

        \path[fill=front1,opacity=\transparent] (b) -- (c5) -- (c4) -- cycle;
        \path[fill=front2,opacity=\transparent] (b) -- (c4) -- (c3) -- cycle;
        \path[fill=front3,opacity=\transparent] (b) -- (c3) -- (c2) -- cycle;
        \path[fill=front4,opacity=\transparent] (b) -- (c2) -- (c1) -- cycle;

        \draw (b) -- (c5) -- (c4) -- (c3) -- (c2) -- (c1) -- (b');
        \draw[very thick] (b') -- (a');
        \draw (b) -- (c4);
        \draw (b) -- (c3);
        \draw (b) -- (c2);
        \draw (b) -- (c1);

        \node[above] at (c3) {$\phantom\ldots$};
    \end{tikzpicture}
    \
    \begin{tikzpicture}[line cap=round,line join=round, thick,baseline=(current bounding box.north)]
        \coordinate (o) at (-90:\r);
        \coordinate (a) at (-67.5:\r);
        \coordinate (b) at (-45:\r);
        \coordinate (c5) at (0:\r);
        \coordinate (c4) at (45:\r);
        \coordinate (c3) at (90:\r);
        \coordinate (c2) at (135:\r);
        \coordinate (c1) at (180:\r);
        \coordinate (b') at (225:\r);
        \coordinate (a') at (247.5:\r);

        \path[fill=root] (o) -- (a) -- (c5) -- cycle;
        \path[fill=front1] (a) -- (b) -- (c5) -- cycle;

        \path[fill=facet2] (o) -- (c5) -- (c4) -- (c3) -- (c2) -- (c1) -- (b') -- cycle;
        \draw[loosely dashed,very thick,opacity=0.25] (o) -- (c4);
        \draw[loosely dashed,very thick,opacity=0.25] (o) -- (c3);
        \draw[loosely dashed,very thick,opacity=0.25] (o) -- (c2);
        \draw[loosely dashed,very thick,opacity=0.25] (o) -- (c1);

        \path[fill=facet1] (o) -- (b') -- (a') -- cycle;
        \draw[loosely dashed,very thick] (o) -- (b');

        \draw[loosely dashed,very thick] (a') -- (o) -- (a);
        \draw[loosely dashed,very thick] (o) -- (c5);
        \draw[dashed,very thick] (a) -- (c5);
        \draw[densely dashed] (a) -- (b);
        \draw[densely dashed] (b') -- (c5);
        \draw[densely dashed] (b') -- (c4);
        \draw[densely dashed] (b') -- (c3);
        \draw[densely dashed] (b') -- (c2);
        \draw (a') arc (-112.5:-45:\r);

        \path[fill=front1,opacity=\transparent] (b) -- (c5) -- (c4) -- cycle;
        \path[fill=front2,opacity=\transparent] (b) -- (c4) -- (c3) -- cycle;
        \path[fill=front3,opacity=\transparent] (b) -- (c3) -- (c2) -- cycle;
        \path[fill=front4,opacity=\transparent] (b) -- (c2) -- (c1) -- cycle;

        \draw (b) -- (c5);
        \draw[very thick] (c5) -- (c4) -- (c3) -- (c2) -- (c1) -- (b') -- (a');
        \draw (b) -- (c4);
        \draw (b) -- (c3);
        \draw (b) -- (c2);
        \draw (b) -- (c1);

        \node[above] at (c3) {$\phantom\ldots$};
    \end{tikzpicture}

    \begin{tikzpicture}[line cap=round, line join=round, thick,baseline=(current bounding box.north)]
        \coordinate (o) at (-90:\r);
        \coordinate (a) at (-67.5:\r);
        \coordinate (b) at (-45:\r);
        \coordinate (c5) at (0:\r);
        \coordinate (c4) at (45:\r);
        \coordinate (c3) at (90:\r);
        \coordinate (c2) at (135:\r);
        \coordinate (c1) at (180:\r);
        \coordinate (b') at (225:\r);
        \coordinate (a') at (247.5:\r);

        \coordinate (d4) at (30:\r);
        \coordinate (e3) at (60:\r);
        \coordinate (d3) at (75:\r);
        \coordinate (e2) at (105:\r);
        \coordinate (d2) at (120:\r);
        \coordinate (e1) at (150:\r);
        \coordinate (d1) at (165:\r);

        \path[fill=root] (o) -- (a) -- (c5) -- cycle;
        \path[fill=front1] (a) -- (b) -- (c5) -- cycle;

        \path[fill=facet2] (o) -- (c5) -- (c4) -- (c3) -- (c2) -- (c1) -- (b') -- cycle;
        \draw[loosely dashed,very thick] (o) -- (c4);
        \draw[loosely dashed,very thick] (o) -- (c3);
        \draw[loosely dashed,very thick] (o) -- (c2);
        \draw[loosely dashed,very thick] (o) -- (c1);

        \path[fill=facet1] (o) -- (b') -- (a') -- cycle;
        \draw[loosely dashed,very thick] (o) -- (b');

        \draw[loosely dashed,very thick] (a') -- (o) -- (a);
        \draw[loosely dashed,very thick] (o) -- (c5);
        \draw[dashed,very thick] (a) -- (c5);
        \draw[densely dashed] (a) -- (b);
        \draw (a') arc (-112.5:-45:\r);

        \path[fill=front4] (c2) -- (e1) -- (d1) -- (c1) -- cycle;
        \draw[dashed,very thick] (c2) -- (c1);
        \draw[densely dashed] (e1) -- (c1);

        \path[fill=front1] (b) -- (c5) -- (d4) -- cycle;
        \path[fill=front2] (b) -- (d4) -- (c4) -- (e3) -- (d3) -- cycle;
        \path[fill=front3] (b) -- (d3) -- (c3) -- (e2) -- (d2) -- cycle;
        \path[fill=front4] (b) -- (d2) -- (c2) -- (e1) -- (d1) -- cycle;

        \begin{scope}
            \clip (b) -- (c5) -- (d4) -- (c4) -- (e3) -- (d3) -- (c3) -- (e2) -- (d2) -- (c2) -- (e1) -- (d1) -- cycle;
        
            \draw[dashed,very thick,opacity=0.25] (c5) -- (c4) -- (c3) -- (c2) -- (c1);
            \draw[densely dashed,opacity=0.25] (e3) -- (c3);
            \draw[densely dashed,opacity=0.25] (e2) -- (c2);
            \draw[densely dashed,opacity=0.25] (e1) -- (c1);
            \draw[loosely dashed,very thick,opacity=0.25] (o) -- (c5);
            \draw[loosely dashed,very thick,opacity=0.25] (o) -- (c4);
            \draw[loosely dashed,very thick,opacity=0.25] (o) -- (c3);
            \draw[loosely dashed,very thick,opacity=0.25] (o) -- (c2);
            \draw[dashed,very thick,opacity=0.25] (a) -- (c5);
        \end{scope}

        \draw (b) -- (c5);
        \draw[very thick] (c1) -- (b') -- (a');
        \draw (b) -- (d4) -- (c4);
        \draw (b) -- (d3) -- (c3);
        \draw (b) -- (d2) -- (c2);
        \draw (b) -- (d1) -- (c1);

        \draw (c4) -- (e3) -- (d3);
        \draw (c3) -- (e2) -- (d2);
        \draw (c2) -- (e1) -- (d1);
        \draw (c5) -- (d4) -- (d3) -- (d2) -- (d1);
        \draw (e3) -- (d4);
        \draw (e2) -- (d3);
        \draw (e1) -- (d2);

        \node[above left=-2pt] at (d1) {$\alpha_1$};
        \node[above left=-2pt] at (e1) {$\beta_1$};
        \node[above] at (d2) {$\alpha_2$};
        \node[above] at (e2) {$\beta_2$};
        \node[above] at (d3) {$\ldots$};
    \end{tikzpicture}
    \
    \tikzmath{\r = 5.17;}
    \begin{tikzpicture}[line cap=round, line join=round, thick,baseline=(current bounding box.north)]
        \coordinate (o) at (-90:\r);
        \coordinate (a) at (-67.5:\r);
        \coordinate (b) at (-45:\r);
        \coordinate (c5) at (0:\r);
        \coordinate (c4) at (45:\r);
        \coordinate (c3) at (90:\r);
        \coordinate (c2) at (135:\r);
        \coordinate (c1) at (180:\r);
        \coordinate (b') at (225:\r);
        \coordinate (a') at (247.5:\r);

        \coordinate (d4) at (30:\r);
        \coordinate (e3) at (60:\r);
        \coordinate (d3) at (75:\r);
        \coordinate (e2) at (105:\r);
        \coordinate (d2) at (120:\r);
        \coordinate (e1) at (150:\r);
        \coordinate (d1) at (165:\r);

        \coordinate (g3) at (65:\r);
        \coordinate (f3) at (70:\r);
        \coordinate (g2) at (110:\r);
        \coordinate (f2) at (115:\r);

        \path[fill=root] (o) -- (a) -- (c5) -- cycle;
        \path[fill=front1] (a) -- (b) -- (c5) -- cycle;

        \path[fill=facet2] (o) -- (c5) -- (c4) -- (c3) -- (c2) -- (c1) -- (b') -- cycle;
        \draw[loosely dashed,very thick] (o) -- (c4);
        \draw[loosely dashed,very thick] (o) -- (c3);
        \draw[loosely dashed,very thick] (o) -- (c2);
        \draw[loosely dashed,very thick] (o) -- (c1);

        \path[fill=facet1] (o) -- (b') -- (a') -- cycle;
        \draw[loosely dashed,very thick] (o) -- (b');

        \draw[loosely dashed,very thick] (a') -- (o) -- (a);
        \draw[loosely dashed,very thick] (o) -- (c5);
        \draw[dashed,very thick] (a) -- (c5);
        \draw[densely dashed] (a) -- (b);
        \draw (a') arc (-112.5:-45:\r);

        \path[fill=front4] (c2) -- (e1) -- (d1) -- (c1) -- cycle;
        \draw[dashed,very thick] (c2) -- (c1);
        \draw[densely dashed] (e1) -- (c1);

        \path[fill=front1] (b) -- (c5) -- (d4) -- cycle;
        \path[fill=front2] (b) -- (d4) -- (c4) -- (e3) -- (d3) -- cycle;
        \path[fill=front3] (b) -- (d3) -- (c3) -- (e2) -- (d2) -- cycle;
        \path[fill=front4] (b) -- (d2) -- (c2) -- (e1) -- (d1) -- cycle;

        \path[fill=complex] (b) -- (d4) -- (e3) -- (g3) -- (f3) -- (d3) -- (e2) -- (g2) -- (f2) -- (d2) -- (e1) -- (d1) -- cycle;
        \draw (e3) -- (g3) -- (f3) -- (d3);
        \draw (e2) -- (g2) -- (f2) -- (d2);

        \draw[densely dashed,opacity=0.25] (e2) -- (d2);
        \draw[densely dashed,opacity=0.25] (e2) -- (f2);
        \draw (e1) -- (f2);
        \draw (e1) -- (g2);

        \draw[densely dashed,opacity=0.25] (e3) -- (d3);
        \draw[densely dashed,opacity=0.25] (e3) -- (f3);
        \draw (e2) -- (f3);
        \draw (e2) -- (g3);

        \begin{scope}
            \clip (b) -- (c5) -- (d4) -- (c4) -- (e3) -- (d3) -- (c3) -- (e2) -- (d2) -- (c2) -- (e1) -- (d1) -- cycle;
        
            \draw[dashed,very thick,opacity=0.25] (c5) -- (c4) -- (c3) -- (c2) -- (c1);
            \draw[densely dashed,opacity=0.25] (e3) -- (c3);
            \draw[densely dashed,opacity=0.25] (e2) -- (c2);
            \draw[densely dashed,opacity=0.25] (e1) -- (c1);
            \draw[loosely dashed,very thick,opacity=0.25] (o) -- (c5);
            \draw[loosely dashed,very thick,opacity=0.25] (o) -- (c4);
            \draw[loosely dashed,very thick,opacity=0.25] (o) -- (c3);
            \draw[loosely dashed,very thick,opacity=0.25] (o) -- (c2);
            \draw[dashed,very thick,opacity=0.25] (a) -- (c5);
        \end{scope}

        \draw (b) -- (c5);
        \draw[very thick] (c1) -- (b') -- (a');
        \draw (b) -- (d4) -- (c4);
        \draw (d3) -- (c3);
        \draw (d2) -- (c2);
        \draw (b) -- (d1) -- (c1);

        \draw (c4) -- (e3);
        \draw (c3) -- (e2);
        \draw (c2) -- (e1) -- (d1);
        \draw (c5) -- (d4);
        \draw (e3) -- (d4);
        \draw (e2) -- (d3);
        \draw (e1) -- (d2);

        \path (e2) -- (b) node[pos=0.33] {$X(s,t-1)$};
    \end{tikzpicture}
    \caption{From left to right and then down, attaching a copy of $X(s,t-1)$ to $X(s-1,K'(s,t-1))$. We have ``zoomed in" on one fillet of $X(s-1,K'(s,t-1))$ and the teeth neighboring it. The numerous other vertices of $X(s-1,K'(s,t-1))$, if pictured, would lie along the bottom arc.}
    \label{fig-xst-recursion}
\end{figure}

\end{itemize}
\end{construction}

The proofs of \ref{property-3ball}-\ref{property-hub} are routine double induction. We sketch the proof of \ref{property-3ball} below and omit \ref{property-circle}-\ref{property-hub}.

We can easily check that $X(s,1)$ is a disc for all $s$ and that $X(1,t)$ is a 3-ball for all $t>1$. In the cases $s,t>1$, recall that we obtain $X(s,t)$ from $X(s-1,K'(s,t-1))$ by repeatedly subdividing faces, attaching a facet along a disc, or attaching a copy of $X(s,t-1)$ along a disc.

If we subdivide faces of a 3-ball, or if we attach two 3-balls along a disc, the result is a 3-ball. Thus, if $X(s,t-1)$ and $X(s-1,K'(s,t-1))$ are both 3-balls, then $X(s,t)$ is a 3-ball and we are done. This leaves only the case $t=2$, in which $X(s-1,K'(s,t-1))$ is a 3-ball and $X(s,t-1)$ is a disc. However, if $t=2$, then the disc along which we attach our two complexes is equal to $X(s,t-1)$ itself; thus, the result is still a 3-ball. It follows that $X(s,t)$ is a 3-ball for all $t>1$.

At last, we use $X(s,t)$ to construct our family of 3-spheres $S(s,t)$. Note that all vertices of $X(s,t)$ lie on the boundary of $X(s,t)$.

\begin{definition}
    For all $s$ and all $t>1$, we define $S(s,t)$ as the 3-sphere obtained by coning over the boundary of $X(s,t)$.
\end{definition}

\subsection{Proofs of shellability}
\label{ss-shellable}
We will now prove that $S(s,t)$ is shellable and dual shellable.

\begin{proposition}
\label{S-shellable}
    $S(s,t)$ is shellable for all $s$ and all $t>1$.
\end{proposition}

\begin{proof}
    First, we will show that for all $s$ and $t$, the facets of $X(s,t)$ can be listed in some order $Y_1, \ldots, Y_n$ such that $Y_i \cap \bigcup_{j=1}^{i-1} Y_j$ is contained in the bottom of $\partial Y_i$ for $i=1,\ldots,n$. Intuitively, this order represents a way of ``stacking" facets from lowest to highest.

    We proceed by double induction on $s$ and $t$. For the base case $s=1,t>1$, we may order the facets of $X(1,t)$ as listed in Construction \ref{ball-construction}. The base case $t=1$ is trivial, as $X(s,1)$ has no facets.

    For the inductive step, fix $s,t>1$, and assume that both $X(s,t-1)$ and $X(s-1,K'(s,t-1))$ admit such an order. Recall Steps 1-5 of Construction \ref{ball-construction}. We begin with the facets of $X(s-1,K'(s,t-1))$ listed in the given order. On Step 2, we prepend the new facet to our list. On Step 3, we do the same. After Step 4, our list still satisfies the required property, as the intersection of any two facets either stays the same or grows from a triangle into a quadrilateral (not changing which facet is on top). Finally, on Step 5, we list the facets of our new copy of $X(s,t-1)$ in the given order and append them to our current list. We do this process for each repetition of Steps 1-5 to obtain the desired list of facets of $X(s,t)$. This completes the inductive step.

    Now, fix $s$ and $t>1$. Let $Y_1,\ldots,Y_n$ be the list of facets of $X(s,t)$ described above. Let $v$ be the new vertex introduced when constructing $S(s,t)$ from $X(s,t)$. Then each facet of $S(s,t)$ containing $v$ is isomorphic to a pyramid with apex $v$ and base contained in $\partial X(s,t)$.

    Let $W$ be the union of all pyramidal facets having apex $v$ and base contained in the bottom of $\partial X(s,t)$. Then $S(s,t)/v$ is a 2-sphere, and $W/v \subset S(s,t)/v$ is a closed disc. Thus, by Lemma \ref{2-sphere-shelling}, there exists a shelling $(T_1/v, \ldots, T_k/v, \ldots, T_m/v)$ of $S(s,t)/v$ such that $T_1, \ldots, T_k$ are the facets of $W$.

    We will use Lemma \ref{3-sphere-shelling}i to prove that
    \[
        ( T_1,\ \ldots,\ T_k,\ Y_1,\ \ldots,\ Y_n,\ T_{k+1},\ \ldots,\ T_m)
    \]
    is a shelling of $S(s,t)$. For $i=1,\ldots,n$, we can see that
    \[
        Y_i \cap \left( \bigcup_{j=1}^k T_j \cup \bigcup_{j=1}^{i-1} Y_j \right)
    \]
    is exactly the bottom of $\partial Y_i$, which is homeomorphic to a disc. Meanwhile, since $(T_1/v, \ldots, T_m/v)$ is a shelling of $S(s,t)/v$, for $i=2, \ldots, m-1$, we know that $T_i/v \cap \bigcup_{j=1}^{i-1}T_j/v$ is a union of consecutive edges of $T_i/v$ homeomorphic to an interval. Thus, $T_i \cap \bigcup_{j=1}^{i-1}T_j$ is a union of consecutive lateral triangles of $T_i$ homeomorphic to a disc.

    Consequently, for $i=k+1, \ldots, m-1$,
    \[
        T_i \cap \left( \bigcup_{j=1}^k T_j \cup \bigcup_{j=1}^n Y_j \cup \bigcup_{j=k+1}^{i-1} T_j \right)
    \]
    is a union of consecutive lateral triangles of $T_i$ homeomorphic to a disc, plus the base of $T_i$. This union is itself homeomorphic to a disc.

    By Lemma \ref{3-sphere-shelling}i, we may conclude that the above is a shelling of $S(s,t)$.
\end{proof}

\begin{proposition}
\label{S-dual-shellable}
    $S(s,t)$ is dual shellable for all $s$ and all $t>1$.
\end{proposition}

\begin{proof}
    Suppose $t>1$. Let $v_1,\ldots,v_{n-1}$ be the vertices of $X(s,t)$ in counterclockwise order, where $v_1$ is the hub of $X(s,t)$. Let $v_n$ be the new vertex introduced when constructing $S(s,t)$ from $X(s,t)$. We will use Lemma \ref{3-sphere-shelling}ii to show that $(v_1,\ldots,v_n)$ is a dual shelling of $S(s,t)$.

    Fix an arbitrary $1<i<n$. We will first prove that $S(s,t)/v_i \cap \bigcup_{j=1}^{i-1}\operatorname{st}v_j$ is a union of vertices of $S(s,t)/v_i$ and their stars. Let $Z$ be a face of $S(s,t)$ such that $v_i,v_j \in Z$ for some $j<i$. List the vertices of $Z$ in index order, and let $v_k$ be the vertex of $Z$ immediately preceding $v_i$, so $j\leq k < i$. If $Z \subseteq X(s,t)$, then by \ref{property-circle}, the image of $Z$ projected onto the page is the convex hull of the images of its vertices. Thus, $v_iv_k$ is an edge of $Z$. If instead $Z \not\subseteq X(s,t)$, then $Z$ is isomorphic to a pyramid with apex $v_n$ and base contained in $X(s,t)$. By applying our previous argument to the base of $Z$, we see again that $v_iv_k$ is an edge of $Z$.

    We have shown that for all $j<i$ and all faces $Z \subseteq S(s,t)$, if $\operatorname{relint}Z \subseteq \operatorname{st}v_i \cap \operatorname{st}v_j$, then $\operatorname{relint}Z \subset \operatorname{st}(v_iv_k)$ for some edge $v_iv_k$ with $j \leq k < i$. Thus, $\operatorname{st}v_i \cap \bigcup_{j=1}^{i-1}\operatorname{st}v_j$ is a union of stars of edges incident to $v_i$. It follows that $S(s,t)/v_i \cap \bigcup_{j=1}^{i-1}\operatorname{st}v_j$ is a union of stars of vertices of $S(s,t)/v_i$.
    
    Second, we will prove that $S(s,t)/v_i \cap \bigcup_{j=1}^{i-1}\operatorname{st}v_j$ is homeomorphic to an open disc. Let $w$ be the vertex of $S(s,t)/v_i$ induced by edge $v_iv_n$. If $i=n-1$, then $S(s,t)/v_i \cap \bigcup_{j=1}^{i-1}\operatorname{st}v_j = S(s,t)/v_i\backslash w$, so we are done.

    Suppose $i<n-1$. In our picture of $X(s,t)$, let $H$ be a plane perpendicular to the page and intersecting the relative interiors of edges $v_{i-1}v_i,v_iv_{i+1}$. Let $H^+$ be the resulting open half-space containing $v_i$. As a consequence of \ref{property-circle}, any $k$-face of $X(s,t)$ containing $v_i$ must intersect $\operatorname{cl}H^+$ on a closed $k$-ball. Thus, by definition, the subdivision of $H \cap X(s,t)$ induced by $X(s,t)$ is exactly the vertex figure $X(s,t)/v_i$. Let $U=X(s,t)/v_i \cap \bigcup_{j=1}^{i-1}(\operatorname{st}v_j \cap X(s,t))$.

    Let $J$ be a second plane perpendicular to the page, intersecting $v_i$ and the relative interior of edge $v_{n-1}v_1$. Let $J^+$ be the resulting open half-space containing $v_1, \ldots, v_{i-1}$, and let $U'=X(s,t)/v_i \cap J^+$. Then $U'$ is homeomorphic to a half-disc with $U' \cap \partial U' = \partial(X(s,t)/v_i) \cap J^+$. For any edge $v_iv_j$ with $j<i$, we have $\operatorname{relint}v_iv_j \subset J^+$, so the vertex $v_iv_j/v_i$ of $X(s,t)/v_i$ is in $U'$.

    Let $Z/v_i$ be a face of $X(s,t)/v_i$. If $\operatorname{relint}(Z/v_i) \subseteq U$, then as previously shown, $v_iv_j/v_i \in Z_i/v_i$ for some edge $v_iv_j$ with $j<i$. Conversely, if $\operatorname{relint}(Z/v_i) \not\subseteq U$, then $Z \cap \{v_1, \ldots, v_{i-1}\} = \varnothing$. In this case, the image of $Z$ on the page is a convex polygon with vertices outside the image of $J^+$; thus, $Z \cap J^+ = \varnothing$, so $Z/v_i \cap U'=\varnothing$. We may conclude that the faces of $X(s,t)/v_i$ with relative interiors contained in $U$ are exactly the faces that nontrivially intersect $U'$. Consequently, $U$ deformation retracts onto $U'$, so $U$ is simply connected.

    We obtain the vertex figure $S(s,t)/v_i$ by coning over the boundary of $X(s,t)/v_i$ at $w$. For each nonempty $k$-face (i.e. edge or vertex) $e$ of $\partial(X(s,t)/v_i)$, let $\hat e$ be the $(k+1)$-face of $S(s,t)/v_i$ containing $e$ and $w$. Then
    \[
        S(s,t)/v_i \cap \bigcup_{j=1}^{i-1}\operatorname{st}v_j = U \cup \bigcup \{\operatorname{relint} \hat e \mid \operatorname{relint}e \subseteq U \cap \partial U \}.
    \]
    It follows that $S(s,t)/v_i \cap \bigcup_{j=1}^{i-1} \operatorname{st}v_j$ deformation retracts onto $U$. Thus, $S(s,t)/v_i \cap \bigcup_{j=1}^{i-1} \operatorname{st}v_j$ is simply connected. We may conclude that $S(s,t)/v_i \cap \bigcup_{j=1}^{i-1} \operatorname{st}v_j$ is homeomorphic to an open disc by the Riemann mapping theorem (see \cite{Luo17}).

    For each $1 < i < n$, we have shown that $S(s,t)/v_i \cap \bigcup_{j=1}^{i-1} \operatorname{st} v_j$ is a union of vertices of $S(s,t)/v_i$ and their stars and is homeomorphic to an open disc. By Lemma \ref{3-sphere-shelling}ii, $(v_1, \ldots, v_n)$ is therefore a dual shelling of $S(s,t)$.
\end{proof}

\subsection{Face asymptotics}
\label{ss-asymptotics}

In this subsection, we will derive recursive formulas for the face numbers of $S(s,t)$ and determine their asymptotic rates of growth. Our formulas will depend on $K(s,t)$, $K'(s,t)$, and one more function that we now define.

Recall that when constructing $X(s,t)$ from $X(s-1,K'(s,t-1))$, each time we attach a copy of $X(s,t-1)$, we introduce exactly two new root triangles that do not correspond with a tooth.

\begin{definition}
\label{R-defn}
    For all positive integers $s$ and $t$, we define $R(s,t)$ as the number of root triangles minus the number of teeth in $X(s,t)$. Equivalently,
    \begin{itemize}
        \item $R(1,t)=0$ for all $t$.
        \item $R(s,1)=0$ for all $s$.
        \item $R(s,t)=2K(s-1,K'(s,t-1))+R(s-1,K'(s,t-1))$ for $s,t>1$.
    \end{itemize}
\end{definition}

\begin{lemma}
\label{R-bounds}
    For all $s,t>1$,
    \[
        \frac{2K(s,t)}{K(s,t-1)} \leq R(s,t) < \frac{3K(s,t)}{K(s,t-1)}.
    \]
\end{lemma}

\begin{proof}
    First, we will show that $R(s,t)<K(s,t)$ for all $s$ and $t$. This is clearly true if $s=1$ or if $t=1$, as $R(s,t)$ then equals zero. The cases $s,t>1$ follow by induction:
    \begin{align*}
        R(s,t) &= 2K(s-1,K'(s,t-1))+R(s-1,K'(s,t-1))\\
        &< 3K(s-1,K'(s,t-1))\\
        &< K(s,t-1)K(s-1,K'(s,t-1))\\
        &= K(s,t).
    \end{align*}
    The inequality $3<K(s,t-1)$ used above is a consequence of Lemma \ref{K-properties}v.

    We have shown that $R(s,t)<K(s,t)$ for all $s$ and $t$. For $s,t>1$, it follows that
    \[
        \frac{2K(s,t)}{K(s,t-1)} = 2K(s-1,K'(s,t-1)) \leq R(s,t) < 3K(s-1,K'(s,t-1)) = \frac{3K(s,t)}{K(s,t-1)}.\qedhere
    \]
\end{proof}

\begin{definition}
    For $i=0,1,2,3$ and positive integers $s$ and $t$, we define functions $F_i(s,t)$ recursively as follows. We denote by $F(s,t)$ the vector $(F_0(s,t),F_1(s,t),F_2(s,t),F_3(s,t))$.
    \begin{itemize}
        \item $F(1,t) = (2t+7,\ 10t+15,\ 16t+12,\ 8t+4)$ for all $t>1$.
        \item $F(s,1) = (2^{s+1}+5,\ 3 \cdot 2^{s+1}+9,\ 2^{s+3} + 8,\ 2^{s+2} + 4)$ for all $s$.
        \item For $s,t>1$,\vspace{-\baselineskip}
    \end{itemize}
        \begin{align*}
            F_0(s,t) &= K(s-1,K'(s,t-1))[F_0(s,t-1) + R(s,t-1)-2] + F_0(s-1,K'(s,t-1)),\\
            F_1(s,t) &= K(s-1,K'(s,t-1))[F_1(s,t-1) + 3K'(s,t-1) + 3R(s,t-1) - 4] + F_1(s-1,K'(s,t-1)),\\
            F_2(s,t) &= K(s-1,K'(s,t-1))[F_2(s,t-1) + 3K'(s,t-1) + 4R(s,t-1) + 1] + F_2(s-1,K'(s,t-1)),\\
            F_3(s,t) &= K(s-1,K'(s,t-1))[F_3(s,t-1) + 2R(s,t-1) + 3] + F_3(s-1,K'(s,t-1)).
        \end{align*}
\end{definition}

\begin{proposition}
\label{face-numbers}
    The $f$-vector of $S(s,t)$ agrees with $F(s,t)$ for all positive integers $s$ and $t>1$.
\end{proposition}

\begin{proof}
    Let $s,t>1$. Recall Steps 1-5 of Construction \ref{ball-construction}, which we repeat for each fillet of $X(s-1,K'(s,t-1))$ when constructing $X(s,t)$. In Steps 1-4, we add the following numbers of faces to $X(s-1,K'(s,t-1))$ and $\partial X(s-1,K'(s,t-1))$.
    \begin{table}[H]
    \centering
        \begin{NiceTabular}{|l l l|}
            \hline
            $i$ & total $i$-faces added & boundary $i$-faces added\\
            \hline
            0 & $2K'(s,t-1)-3$ & $2K'(s,t-1)-3$ \\
            1 & $7K'(s,t-1)-10$ & $6K'(s,t-1)-10$ \\
            2 & $5K'(s,t-1)-5$ & $4K'(s,t-1)-7$ \\
            3 & $2$ & \\
            \hline
        \end{NiceTabular}
        \caption{Faces added to $X(s-1,K'(s,t-1))$ in Steps 1-4.}
    \end{table}

    In Step 5, we attach a copy of $X(s,t-1)$ along a disc $D$ with
    \begin{align*}
        f(D) &= (2K'(s,t-1)-R(s,t-1)-2,\ 4K'(s,t-1)-2R(s,t-1)-7,\ 2K'(s,t-1)-R(s,t-1)-4),\\
        f(\partial D) &= (2K'(s,t-1)-R(s,t-1)-2,\ 2K'(s,t-1)-R(s,t-1)-2).
    \end{align*}
    If $t>2$, we thus add the following numbers of faces to $X(s-1,K'(s,t-1))$ and $\partial X(s-1,K'(s,t-1))$.
    \begin{table}[H]
    \setlength{\tabcolsep}{5pt}
    \centering
        \begin{NiceTabular}{|l l l|}
            \hline
            $i$ & total $i$-faces added & boundary $i$-faces added\\
            \hline
            0 & $f_0(X(s,t-1))-2K'(s,t-1)+R(s,t-1)+2$ & $f_0(\partial X(s,t-1))-2K'(s,t-1)+R(s,t-1)+2$ \\
            1 & $f_1(X(s,t-1))-4K'(s,t-1)+2R(s,t-1)+7$ & $f_1(\partial X(s,t-1))-6K'(s,t-1)+3R(s,t-1)+12$ \\
            2 & $f_2(X(s,t-1))-2K'(s,t-1)+R(s,t-1)+4$ & $f_2(\partial X(s,t-1))-4K'(s,t-1)+2R(s,t-1)+8$ \\
            3 & $f_3(X(s,t-1))$ & \\
            \hline
        \end{NiceTabular}
        \caption{Faces added to $X(s-1,K'(s,t-1))$ in Step 5.}
        \label{tbl-step-5}
    \end{table}

    The case $t=2$ is slightly different: every face of $X(s,1)$ is identified with an existing face, so we do not add any faces in Step 5. The middle column of Table \ref{tbl-step-5} is still consistent, as
    \[
        f(X(s,1)) = (2^{s+1}+4,\ 2^{s+2}+5,\ 2^{s+1}+2,\ 0) = (2K'(s,1)-2,\ 4K'(s,1)-7,\ 2K'(s,1)-4,\ 0),
    \]
    making all values in the middle column zero. However, the right column assumes that interior faces of $D$ will be in the interior of our 3-complex after identification, which is only true for $t>2$. Nevertheless, we can make the right column consistent with $t=2$ by redefining $f(\partial X(s,1))$ as
    \[
        f(\partial X(s,1)) = (2^{s+1}+4,\ 3\cdot 2^{s+1}+6,\ 2^{s+2}+4) = (2K'(s,1)-2,\ 6K'(s,1)-12,\ 4K'(s,1)-8),
    \]
    making all values in the right column zero. This is what $f(\partial X(s,1))$ will mean for the remainder of our proof; the abuse of notation is harmless, as $S(s,t)$ is only defined for $t>1$.

    For all $s,t>1$, we have added the following numbers of faces to $X(s-1,K'(s,t-1))$ and $\partial X(s-1,K'(s,t-1))$ in Steps 1-5.
    \begin{table}[H]
    \centering
        \begin{NiceTabular}{|l l l|}
            \hline
            $i$ & total $i$-faces added & boundary $i$-faces added\\
            \hline
            0 & $f_0(X(s,t-1))+R(s,t-1)-1$ & $f_0(\partial X(s,t-1))+R(s,t-1)-1$ \\
            1 & $f_1(X(s,t-1))+3K'(s,t-1)+2R(s,t-1)-3$ & $f_1(\partial X(s,t-1))+3R(s,t-1)+2$ \\
            2 & $f_2(X(s,t-1))+3K'(s,t-1)+R(s,t-1)-1$ & $f_2(\partial X(s,t-1))+2R(s,t-1)+1$ \\
            3 & $f_3(X(s,t-1))+2$ & \\
            \hline
        \end{NiceTabular}
        \caption{Faces added to $X(s-1,K'(s,t-1))$ in Steps 1-5.}
        \label{faces-added}
    \end{table}

    For all $s$ and $t$, let
    \begin{align*}
        F_0'(s,t) &= f_0(X(s,t)) + 1,\\
        F_1'(s,t) &= f_1(X(s,t)) + f_0(\partial X(s,t)),\\
        F_2'(s,t) &= f_2(X(s,t)) + f_1(\partial X(s,t)),\\
        F_3'(s,t) &= f_3(X(s,t)) + f_2(\partial X(s,t)),
    \end{align*}
    and $F'(s,t)=(F_0'(s,t),F_1'(s,t),F_2'(s,t),F_3'(s,t)).$ Recall that for $t>1$, we obtain the sphere $S(s,t)$ by coning over the boundary of $X(s,t)$, so $f(S(s,t))=F'(s,t)$.
    
    We can observe that $f(X(1,t)) = (2t+6,8t+9,10t,4t-4)$ and $f(\partial X(1,t)) = (2t+6,6t+12,4t+8)$ for all $t>1$, so
    \[
        F'(1,t) = (2t+7,\ 10t+15,\ 16t+12,\ 8t+4) = F(1,t).
    \]
    As previously discussed, for all $s$, we have $f(X(s,1)) = (2^{s+1}+4,2^{s+2}+5,2^{s+1}+2,0)$ and $f(\partial X(s,1)) = (2^{s+1}+4,3\cdot 2^{s+1}+6,2^{s+2}+4)$. Thus,
    \[
        F'(s,1) = (2^{s+1}+5,\ 3\cdot 2^{s+1}+9,\ 2^{s+3}+8,\ 2^{s+2}+4) = F(s,1).
    \]

    For $s,t>1$, we use Table \ref{faces-added} to obtain
    \begin{align*}
        F_0'(s,t) &= K(s-1,K'(s,t-1))[F_0'(s,t-1) + R(s,t-1)-2] + F_0'(s-1,K'(s,t-1)),\\
        F_1'(s,t) &= K(s-1,K'(s,t-1))[F_1'(s,t-1) + 3K'(s,t-1) + 3R(s,t-1) - 4] + F_1'(s-1,K'(s,t-1)),\\
        F_2'(s,t) &= K(s-1,K'(s,t-1))[F_2'(s,t-1) + 3K'(s,t-1) + 4R(s,t-1) + 1] + F_2'(s-1,K'(s,t-1)),\\
        F_3'(s,t) &= K(s-1,K'(s,t-1))[F_3'(s,t-1) + 2R(s,t-1) + 3] + F_3'(s-1,K'(s,t-1)).
    \end{align*}
    These match the recursive formulas for $F(s,t)$. Thus, $F'(s,t)=F(s,t)$ for all $s$ and $t$. We may conclude that the $f$-vector of $S(s,t)$ agrees with $F(s,t)$ for all $s$ and $t>1$.    
\end{proof}

\begin{proposition}
\label{face-asymptotics}
    For all $s$ and $t$ with $t>s$,
    \begin{align}
        1 < \frac{F_0(s,t)}{tK(s,t)} &< \frac{F_3(s,t)}{tK(s,t)} < 72,\label{F03-bounds}\\
        \frac{3s}{2} < \frac{F_1(s,t)}{tK(s,t)} &<  \frac{F_2(s,t)}{tK(s,t)} < 158s.\label{F12-bounds}
    \end{align}
\end{proposition}

\begin{proof}
    Observe that for all $t>1$ and for all $s$, respectively,
    \begin{align*}
        \frac{F(1,t)}{tK(1,t)} &= \left( 1+\frac{7}{2t},\ 5+\frac{15}{2t},\ 8+\frac{6}{t},\ 4+\frac{2}{t} \right),\\
        \frac{F(s,1)}{K(s,1)} &= (2+5\cdot2^{-s},\ 6+9\cdot2^{-s},\ 8+2^{3-s},\ 4+2^{2-s}).
    \end{align*}

    For $s,t>1$, we may rearrange the formula for $F_0(s,t)$ to obtain
    \[
        \frac{F_0(s,t)}{K(s,t)} - \frac{F_0(s,t-1)}{K(s,t-1)} = \frac{R(s,t-1)-2}{K(s,t-1)} + \frac{F_0(s-1,K'(s,t-1))}{K(s,t)}.
    \]
    Substituting the dummy variable $u$ for $t-1$ and summing from $u=1$ to $u=t-1$, we get
    \[
        \frac{F_0(s,t)}{K(s,t)} - \frac{F_0(s,1)}{K(s,1)} = \sum_{u=1}^{t-1} \frac{R(s,u)-2}{K(s,u)} + \sum_{u=1}^{t-1} \frac{F_0(s-1,K'(s,u))}{K(s,u+1)}.
    \]
    Equivalently,
    \begin{equation}
        \frac{F_0(s,t)}{K(s,t)} = 2+5\cdot2^{-s} + \sum_{u=1}^{t-1}\frac{R(s,u)-2}{K(s,u)} + \sum_{u=1}^{t-1} \frac{K'(s,u)}{K(s,u)} \cdot \frac{F_0(s-1,K'(s,u))}{K'(s,u)K(s-1,K'(s,u))}.\label{F0-formula}
    \end{equation}

    Through similar manipulations, we find that for all $s,t>1$,
    \begin{align}
        \frac{F_1(s,t)}{K(s,t)} &= 6+9\cdot2^{-s} + \sum_{u=1}^{t-1} \frac{3R(s,u)-4}{K(s,u)} + \sum_{u=1}^{t-1} \frac{K'(s,u)}{K(s,u)} \left[ 3 + \frac{F_1(s-1,K'(s,u))}{K'(s,u)K(s-1,K'(s,u))} \right],\label{F1-formula}\\
        \frac{F_2(s,t)}{K(s,t)} &= 8+2^{3-s} + \sum_{u=1}^{t-1} \frac{4R(s,u)+1}{K(s,u)} + \sum_{u=1}^{t-1} \frac{K'(s,u)}{K(s,u)} \left[ 3 + \frac{F_2(s-1,K'(s,u))}{K'(s,u)K(s-1,K'(s,u))} \right],\label{F2-formula}\\
        \frac{F_3(s,t)}{K(s,t)} &= 4+2^{2-s} + \sum_{u=1}^{t-1}\frac{2R(s,u)+3}{K(s,u)} + \sum_{u=1}^{t-1} \frac{K'(s,u)}{K(s,u)} \cdot \frac{F_3(s-1,K'(s,u))}{K'(s,u)K(s-1,K'(s,u))}.\label{F3-formula}
    \end{align}
    Note by Lemma \ref{K-properties}v that $K'(s,u)>s-1$ for all $u$. For the remainder of our proof, when working with expressions of the form $F_i(s-1,K'(s,u))$, we will use this inequality without explicit reference.
    
    We begin our proof of (\ref{F03-bounds}). It is clear that $F_3(s,t)>F_0(s,t)$ for all $s$ and $t$, so we need only prove that $F_0(s,t)/(tK(s,t))>1$ and $F_3(s,t)/(tK(s,t))<72$ for $t>s$. We will prove these in order.

    We prove that $F_0(s,t)/(tK(s,t))>1$ by induction on $s$. The base case $s=1$ holds, as $F_0(1,t)/(tK(1,t))\allowbreak = 1+7/(2t) > 1$ for all $t>1$.
    
    For the inductive step, let $s>1$, and suppose $F_0(s-1,t)/(tK(s-1,t))>1$ for all $t>s-1$. Recall that $K'(s,u)/K(s,u)>1$ for all $u$ by Lemma \ref{K-properties}iv. Thus, by (\ref{F0-formula}) and the inductive hypothesis, for all $t>1$,
    \[
        \frac{F_0(s,t)}{K(s,t)} > 5\cdot2^{-s} + \sum_{u=1}^{t-1}\frac{R(s,u)-2}{K(s,u)} + t+1.
    \]
    It is easy to verify by Lemma \ref{K-properties}v that $\sum_{u=1}^{t-1}2/K(s,u) < 1.$ Thus, $F_0(s,t)/K(s,t) > t$. It follows that $F_0(s,t)/(tK(s,t))>1$. This completes our proof that $F_0(s,t)/(tK(s,t))>1$ for $t>s$.

    Next, we prove by induction on $s$ that
    \[
        \frac{F_3(s,t)}{tK(s,t)} \leq \prod_{i=1}^s (1+2^{3-i}).
    \]
    Note that $\prod_{i=1}^\infty (1+2^{3-i}) \approx 71.527$, so the above inequality implies that $F_3(s,t)/(tK(s,t))<72$. The base case $s=1$ is immediate, as $F_3(1,t)/(tK(1,t)) = 4+2/t \leq 5$ for all $t>1$.

    For the inductive step, let $s>1$, and suppose $F_3(s-1,t)/(tK(s-1,t)) \leq \prod_{i=1}^{s-1}(1+2^{3-i})$ for all $t>s-1$. By Lemmas \ref{K-properties}v and \ref{R-bounds}, we have
    \[
        \sum_{u=1}^{t-1}\frac{2R(s,u)+3}{K(s,u)} < \sum_{u=1}^{t-2} \frac{6}{K(s,u)} + \sum_{u=1}^{t-1} \frac{3}{K(s,u)} < \frac{9}{2}.
    \]
    Meanwhile, by Lemma \ref{K-quotient} and the inductive hypothesis, for all $u$,
    \[
        \frac{K'(s,u)}{K(s,u)}\cdot\frac{F_3(s-1,K'(s,u))}{K'(s,u)K(s-1,K'(s,u))} < \prod_{i=1}^s(1+2^{3-i}).
    \]
    Thus, by (\ref{F3-formula}),
    \begin{align*}
        \frac{F_3(s,t)}{K(s,t)} &< \frac{19}{2} + (t-1)\prod_{i=1}^s (1+2^{3-i})\\
        &< t\prod_{i=1}^s (1+2^{3-i}).
    \end{align*}
    It follows that $F_3(s,t)/(tK(s,t)) < \prod_{i=1}^s (1+2^{3-i})$. This completes our proof that $F_3(s,t)/(tK(s,t)) \leq \prod_{i=1}^s(1+2^{3-i})$ for $t>s$, which in turn completes our proof of (\ref{F03-bounds}).

    We now begin our proof of (\ref{F12-bounds}). It is clear that $F_2(s,t) > F_1(s,t)$ for all $s$ and $t$, so we need only prove that $F_1(s,t)/(tK(s,t)) > 3s/2$ and $F_2(s,t)/(tK(s,t)) < 158s$ for $t>s$. We will prove these in order.

    We prove that $F_1(s,t)/(tK(s,t)) > 3s/2$ by induction on $s$. The base case $s=1$ holds, as $F_1(1,t)/(tK(1,t))\allowbreak = 5+15/(2t) > 3/2$ for all $t>1$.

    For the inductive step, let $s>1$, and suppose $F_1(s-1,t)/(tK(s-1,t))>3(s-1)/2$ for all $t>s-1$. We can verify by Lemma \ref{K-properties}v that $\sum_{u=1}^{t-1}4/K(s,u)<2$, so in (\ref{F1-formula}), the first three terms of the right-hand side must sum to a positive number. Recall again that $K'(s,u)/K(s,u)>1$ for all $u$. Thus, by (\ref{F1-formula}) and the inductive hypothesis, for all $t>s$, we have
    \begin{align*}
        \frac{F_1(s,t)}{K(s,t)} &> \sum_{u=1}^{t-1}\left[ 3+\frac{F_1(s-1,K'(s,u))}{K'(s,u)K(s-1,K'(s,u))} \right]\\
        &> (t-1)\left[ 3+\frac{3(s-1)}{2} \right]\\
        &= \frac{3(s+1)(t-1)}{2}\\
        &\geq \frac{3st}{2}.
    \end{align*}
    It follows that $F_1(s,t)/(tK(s,t)) > 3s/2$. This completes our proof that $F_1(s,t)/(tK(s,t))\allowbreak > 3s/2$ for $t>s$.

    Next, we prove by induction on $s$ that
    \[
        \frac{F_2(s,t)}{tK(s,t)} \leq 11s\prod_{i=2}^s (1+2^{3-i}).
    \]
    Note that $11\prod_{i=2}^\infty (1+2^{3-i}) \approx 157.359$, so the above inequality implies that $F_2(s,t)/(tK(s,t)) < 158s$. The base case $s=1$ is immediate, as $F_2(1,t)/(tK(1,t)) = 8+6/t \leq 11$ for all $t>1$.

    For the inductive step, let $s>1$, and suppose $F_2(s-1,t)/(tK(s-1,t)) \leq 11(s-1)\prod_{i=2}^{s-1}(1+2^{3-i})$ for all $t>s-1$. By Lemmas \ref{K-properties}v and \ref{R-bounds}, we have
    \[
        \sum_{u=1}^{t-1} \frac{4R(s,u)+1}{K(s,u)} < \sum_{u=1}^{t-2} \frac{12}{K(s,u)} + \sum_{u=1}^{t-1} \frac{1}{K(s,u)} < \frac{13}{2}.
    \]
    Meanwhile, by Lemma \ref{K-quotient} and the inductive hypothesis, for all $u$,
    \begin{align*}
        \frac{K'(s,u)}{K(s,u)}\left[ 3+\frac{F_2(s-1,K'(s,u))}{K'(s,u)K(s-1,K'(s,u))} \right] &< (1+2^{3-s})\left[ 3+11(s-1)\prod_{i=2}^{s-1} (1+2^{3-i}) \right]\\
        &< 11s(1+2^{3-s})\prod_{i=2}^{s-1}(1+2^{3-i})\\
        &= 11s\prod_{i=2}^s (1+2^{3-i}).
    \end{align*}
    Thus, by (\ref{F2-formula}),
    \begin{align*}
        \frac{F_2(s,t)}{K(s,t)} &< \frac{33}{2} + 11s(t-1)\prod_{i=2}^{s} (1+2^{3-i})\\
        &< 11st\prod_{i=2}^{s} (1+2^{3-i}).
    \end{align*}
    It follows that $F_2(s,t)/(tK(s,t)) < 11s\prod_{i=2}^s(1+2^{3-i})$. This completes our proof that $F_2(s,t)/(tK(s,t)) \leq 11s\prod_{i=2}^s(1+2^{3-i})$ for $t>s$, which in turn completes our proof of (\ref{F12-bounds}).
\end{proof}

\subsection{Main result}
\label{ss-main-result}

We now prove our main result.

\begin{theorem}
    For arbitrarily large $n$, there exist shellable, dual shellable 3-spheres $S$ such that
    \[
        f(S) = (\Theta(n),\Theta(n\alpha(n)),\Theta(n\alpha(n)),\Theta(n)).
    \]
\end{theorem}

\begin{proof}
\label{sphere-result}
    For arbitrary $s>1$, let $n=sK(s-1,s)$, and let $S=S(s-1,s)$. By Propositions \ref{S-shellable}-\ref{S-dual-shellable}, $S$ is shellable and dual shellable.

    By Propositions \ref{face-numbers}-\ref{face-asymptotics},
    \begin{align*}
        f_0(S),f_3(S) &= \Theta(sK(s-1,s)),\\
        f_1(S),f_2(S) &= \Theta(s^2K(s-1,s)).
    \end{align*}
    By Lemma \ref{K-Ackermann}, we have $s=\Theta(\alpha(n))$. Thus,
    \[
        f(S) = (\Theta(n),\Theta(n\alpha(n)),\Theta(n\alpha(n)),\Theta(n)).\qedhere
    \]
\end{proof}

\subsection{Further questions}
\label{ss-questions}

We have constructed arbitrarily fat, shellable, dual shellable 3-spheres. This leaves us with a key question: are they polytopes?

\begin{observation}
    $S(1,t)$ is isomorphic to the boundary complex of a 4-polytope for all $t>1$.
\end{observation}

Figure \ref{fig-sage} gives SageMath code for a 4-polytope realizing $S(1,2)$. The ten vertices of $X(1,2)$ are arranged roughly as in Figure \ref{fig-x1t} (left), taking the first three coordinates respectively as east, north, and ``up" out of the page, and ignoring the fourth coordinate. Label vertices as follows:
\[
    v = \textstyle\left(-\frac{4}{5},-\frac{3}{5},1,12\right),\quad a = \textstyle\left(-1,0,\frac{4}{3},0\right),\quad w_2 = \textstyle\left(\frac{4}{5},\frac{3}{5},\frac{1}{3},\frac{3}{5}\right),\quad x_2 = \textstyle\left(-\frac{5}{13},\frac{12}{13},\frac{8}{13},\frac{16}{13}\right).
\]
Here $v$ is the unique vertex disjoint from $X(1,2)$, $a$ is the root of $X(1,2)$ clockwise-adjacent to the hub, and $w_2$ and $x_2$ are the initial vertices of each fillet. For each successive $t>2$, we may obtain a polytope realizing $S(1,t)$ by placing new vertices $w_t$ and $x_t$ beyond triangles $\bigtriangleup avw_{t-1}$ and $\bigtriangleup avx_{t-1}$, respectively, and taking the convex hull.

\begin{figure}[t]
\centering
\lstset{breaklines=true,basicstyle = \ttfamily,columns=fullflexible}
\begin{lstlisting}
sage: P = Polyhedron(vertices=[[0,-1,0,0], [1,0,0,0], [12/13,5/13,3/13,6/13], [4/5,3/5,1/3,3/5], [3/5,4/5,0,0], [5/13,12/13,1/3,15/13], [-5/13,12/13,8/13,16/13], [-3/5,4/5,2/5,0], [-4/5,3/5,9/10,8/5], [-1,0,4/3,0], [-4/5,-3/5,1,12]])
\end{lstlisting}
\caption{SageMath code for a 4-polytope on eleven vertices with boundary isomorphic to $S(1,2)$. All vertices project onto the unit circle in the first two coordinates.}
\label{fig-sage}
\end{figure}

\begin{conjecture}
\label{conj-polytopal}
    For infinitely many $s$, there exists $t>s$ such that $S(s,t)$ is isomorphic to the boundary complex of a 4-polytope.
\end{conjecture}

If one believes in arbitrarily fat 4-polytopes, then we propose $S(s,t)$ as plausible candidates. Intuitively, we imagine a version of Construction \ref{ball-construction} where every step yields a subcomplex of the boundary of a 4-polytope. At each recursion, we have many degrees of freedom when adding vertices in Step 4, especially if we relax the requirement that all vertices project onto the unit circle. Steps 3 and 5 are tricky---3 requires many vertices to lie on a common hyperplane, and 5 requires two complexes to fit together exactly---but we suspect that suitable choices in Step 4 make the rest of the construction possible.

There has been progress on finding polytopal realizations of 3-spheres algorithmically \cite{brinkmann18}\cite{firsching20}\cite{gouveia23}. However, while we conjecture that realizations of many $S(s,t)$ exist, we are not optimistic that they can be found by computer search. The sphere $S(3,4)$ already has more than $2^{2^{2^{15}}}$ vertices, putting it far beyond the scope of relevant algorithms.

Finally, we remark that ours is one of many related constructions for fat 3-spheres. One could modify Construction \ref{ball-construction} by subdividing some facets or splitting some vertices without changing the $f$-vector asymptotically. Where we add one vertex to obtain $S(s,t)$ from $X(s,t)$, one could add many more vertices and still keep fatness roughly the same. Such modifications could yield further candidates for realizability as fat 4-polytopes.


\section{Acknowledgments}
The author would like to thank Isabella Novik for her constant support in writing this paper and for her thorough editing feedback. The author was  partially supported by a graduate fellowship from NSF grant DMS-2246399.


\printbibliography

\end{document}